\documentclass[11pt,twoside]{amsart}

\usepackage{latexsym}
\usepackage{amssymb}
\usepackage{vmargin}
\usepackage{amscd}
\usepackage{stmaryrd}
\usepackage{euscript}
\usepackage{mathrsfs} 
\usepackage[all]{xy}

\usepackage{xr}

\setmargins{32mm}{20mm}{14.6cm}{22cm}{1cm}{1cm}{1cm}{1cm}

\newcommand\ten{\otimes}

\newcommand\eps{\epsilon}

\newcommand\CC{\mathrm{C}}

\newcommand\CCC{\mathrm{CC}}
\newcommand\GT{\mathrm{GT}}
\newcommand\Levi{\mathrm{Levi}}

\renewcommand\H{\mathrm{H}}
\newcommand\z{\mathrm{Z}}

\newcommand\HH{\mathrm{HH}}

\newcommand\Z{\mathbb{Z}}
\newcommand\Q{\mathbb{Q}}

\newcommand\bG{\mathbb{G}}

\newcommand\cB{\mathcal{B}}

\newcommand\cD{\mathcal{D}}

\newcommand\cP{\mathcal{P}}

\newcommand\cPer{\mathcal{P}\!\mathit{er}}

\renewcommand\O{\mathscr{O}}

\newcommand\sD{\mathscr{D}}

\newcommand\sL{\mathscr{L}}

\newcommand\sO{\mathscr{O}}

\newcommand\fX{\mathfrak{X}}
\newcommand\fY{\mathfrak{Y}}

\renewcommand\L{\Lambda}

\newcommand\g{\mathfrak{g}}

\renewcommand\hom{\mathscr{H}\!\mathit{om}}
\newcommand\cHom{\mathcal{H}\!\mathit{om}}

\newcommand\CAlg{\mathrm{CAlg}}

\newcommand\Ass{\mathrm{Ass}}

\newcommand\Hom{\mathrm{Hom}}

\newcommand\map{\mathrm{map}}

\newcommand\HHom{\underline{\mathrm{Hom}}}

\newcommand\cone{\mathrm{cone}}

\newcommand\dg{\mathrm{dg}}
\newcommand\per{\mathrm{per}}

\newcommand{\brh}{\llbracket \hbar \rrbracket}
\newcommand{\brhh}{\llbracket \hbar^2 \rrbracket}

\newcommand\Spec{\mathrm{Spec}\,}

\newcommand\Loc{\mathrm{Loc}}
\newcommand\Set{\mathrm{Set}}

\newcommand\Aff{\mathrm{Aff}}

\newcommand\Sp{\mathrm{Sp}}
\newcommand\PreSp{\mathrm{PreSp}}

\newcommand\Pol{\mathrm{Pol}}

\newcommand\Comp{\mathrm{Comp}}

\newcommand\nondeg{\mathrm{nondeg}}

\newcommand\ad{\mathrm{ad}}

\newcommand\Lim{\varprojlim}

\newcommand\ho{\mathrm{ho}\!}

\newcommand\abuts{\implies}
\newcommand\xra{\xrightarrow}

\newcommand\bt{\bullet}
\newcommand\by{\times}

\newcommand\mc{\mathrm{MC}}
\newcommand\mmc{\underline{\mathrm{MC}}}

\newcommand\Perf{\mathrm{Perf}}

\newcommand\Symm{\mathrm{Symm}}

\newcommand\et{\acute{\mathrm{e}}\mathrm{t}}

\newcommand\Tot{\mathrm{Tot}\,}

\newcommand\pd{\partial}

\newcommand\half{\frac{1}{2}}

\newcommand\gr{\mathrm{gr}}

\newcommand\red{\mathrm{red}}

\newcommand\DR{\mathrm{DR}}

\newcommand\op{\mathrm{opp}}
\newcommand\opp{\mathrm{opp}}

\newcommand\co{\colon\thinspace}

\newcommand\oR{\mathbf{R}}

\newcommand\oL{\mathbf{L}}

\newcommand\uleft\underleftarrow
\newcommand\uline\underline
\newcommand\uright\underrightarrow

\newtheorem{theorem}{Theorem}[section]
\newtheorem{proposition}[theorem]{Proposition}

\newtheorem{lemma}[theorem]{Lemma}
\newtheorem*{theorem*}{Theorem}
\newtheorem*{proposition*}{Proposition}
\newtheorem*{corollary*}{Corollary}
\newtheorem*{lemma*}{Lemma}
\newtheorem*{conjecture*}{Conjecture}

\theoremstyle{definition}
\newtheorem{definition}[theorem]{Definition}
\newtheorem*{definition*}{Definition}

\newtheorem*{notation*}{Notation}

\theoremstyle{remark}
\newtheorem{example}[theorem]{Example}
\newtheorem{examples}[theorem]{Examples}
\newtheorem{remark}[theorem]{Remark}
\newtheorem{remarks}[theorem]{Remarks}

\newtheorem{assumption}[theorem]{Assumption}

\newtheorem*{example*}{Example}
\newtheorem*{examples*}{Examples}
\newtheorem*{remark*}{Remark}
\newtheorem*{remarks*}{Remarks}
\newtheorem*{exercise*}{Exercise}
\newtheorem*{property*}{Property}
\newtheorem*{properties*}{Properties}

\begin{document}

\begin{abstract}
We prove that every $0$-shifted symplectic structure on a derived Artin $n$-stack   admits a curved $A_{\infty}$ deformation quantisation. The classical method of quantising smooth varieties via quantisations of affine space does not apply in this setting, so we develop a new approach. We construct a map from   DQ algebroid quantisations of unshifted symplectic structures on a  derived Artin $n$-stack to power series in  de Rham cohomology, depending only on a choice of Drinfeld associator. This gives an equivalence between even power series and  certain involutive  quantisations, which yield anti-involutive curved $A_{\infty}$ deformations of the dg category of perfect complexes. 
 In particular, there is a canonical quantisation associated to every  symplectic structure on such a stack, which agrees for smooth varieties with the Kontsevich--Tamarkin quantisation for even associators. 
\end{abstract}

\title[Quantisation for unshifted symplectic structures on stacks]
{Deformation quantisation for unshifted symplectic structures on derived Artin stacks}
\author{J.P.Pridham}
\thanks{This work was supported by  the Engineering and Physical Sciences Research Council [grant number EP/I004130/2].}



\maketitle

\section*{Introduction}

For $n >0$, existence of quantisations of $n$-shifted Poisson structures is a formality, following from the equivalence $E_{n+1}\simeq P_{n+1}$ of operads. Quantisations of positively shifted symplectic structures thus follow immediately from the equivalence in \cite{poisson, CPTVV} between symplectic and non-degenerate Poisson structures. In \cite{DQvanish}, quantisation for non-degenerate $(-1)$-shifted Poisson structures was established, and we now consider the $n=0$ case, fleshing out the details sketched in  \cite[\S \ref{DQvanish-nonnegsn}]{DQvanish}.

Beyond the setting of smooth Deligne--Mumford stacks, unshifted symplectic structures only arise on objects incorporating both stacky and derived structures, as non-degeneracy of the symplectic form implies that the cotangent complex must have both positive and negative terms. Examples of such symplectic derived stacks include the derived moduli stack of perfect complexes on an algebraic $K3$ surface, or the derived moduli stack of locally constant $G$-torsors on a compact oriented topological surface, for an algebraic  group $G$ equipped with a Killing form on its Lie algebra. In the latter example, the symplectic structure on the smooth locus is that of \cite{goldmanInvartFns}. 

The common feature in the construction of deformation quantisations for manifolds \cite{DeWildeLecomte,Fedosov, deligneDefFnsSymplectic,kontsevichPoisson,tamarkinOperadicKontsevichFormality} and for smooth algebraic varieties \cite{BezrukavnikovKaledin,kontsevichDQAlgVar,YekutieliDQAG,vdBerghGlobalDQ} is the reduction (\'etale) locally to affine space. For derived Artin stacks, this is not an option, so we develop a new approach to show  that all non-degenerate Poisson structures can be quantised even if the Hochschild complex is not formal. This works by a similar  mechanism to the quantisation of non-degenerate $(-1)$-shifted Poisson structures   in \cite{DQvanish}, combined with formality of the $E_2$ operad.

The proof in \cite{poisson} of the correspondence between $n$-shifted symplectic and non-degenerate Poisson structures relied on the existence, for all Poisson structures $\pi$, of a CDGA morphism $\mu(-,\pi)$ from the de Rham algebra to the algebra $T_{\pi}\widehat{\Pol}(X,n)$  of shifted polyvectors with differential twisted by $\pi$. In \cite{DQvanish}, this idea was extended to establish  the 
existence of quantisations for $(-1)$-shifted symplectic structures, with $\mu$ being an $A_{\infty}$-morphism from the de Rham algebra to the ring of differential operators.  

In order to adapt these constructions to  $0$-shifted symplectic structures, we replace polyvectors or differential operators with the Hochschild complex $\CCC^{\bt}_R(X)$ of a derived Artin stack $X$, defined in terms of a resolution by stacky CDGAs (commutative bidifferential bigraded algebras). Since this has an $E_2$-algebra structure, a choice $w$ of Levi decomposition for the Grothendieck--Teichm\"uller group gives it a $P_2$-algebra structure. Quantisations $\Delta$ are defined as certain Maurer--Cartan elements  $\Delta \in \CCC^{\bt}_R(X)\llbracket \hbar \rrbracket$; these give rise to curved deformations of the dg category of perfect complexes.

 Each quantisation $\Delta$ then defines a morphism $\mu_w(-, \Delta)$ from the de Rham complex $\DR(X)$ to $\CCC^{\bt}_R(X)\llbracket \hbar \rrbracket$ twisted by $\Delta$. In more detail, 
since $[\Delta,-]$ defines a derivation from $\sO_X$ to  $\CCC^{\bt}_R(X)\llbracket \hbar \rrbracket$, it determines a map $\Omega^1_X \to \CCC^{\bt}_R(X)\llbracket \hbar \rrbracket[1]$ and    $\mu_w(-,\Delta)$ is the resulting morphism of CDGAs.
This gives rise to a notion of compatibility between $E_1$ quantisations $\Delta$ and generalised pre-symplectic structures (power series $\omega$ of elements of the de Rham complex): we say that $\omega$ and $\Delta$ are $w$-compatible if
\[
 \mu_w(\omega, \Delta) \simeq \hbar^2  \frac{\pd \Delta}{\pd \hbar}. 
\]

Proposition \ref{QcompatP1} shows that every non-degenerate  quantisation $\Delta$ of a stacky CDGA $A$  has a unique $w$-compatible generalised pre-symplectic structure, thus giving us a map
\[
 Q\cP(A,0)^{\nondeg} \to \H^2(F^2\DR(A)) \by \hbar\H^2(F^1\DR(A))\by \hbar^2\H^1(\DR(A)) \llbracket \hbar\rrbracket
\]
on the space of non-degenerate $0$-shifted $E_1$ quantisations of $A$. 

Moreover, we have spaces $Q\cP(A,0)/G^{k+1}$ consisting of $E_1$ quantisations of order $k$, 
by which we mean Maurer--Cartan elements in $\prod_{j \ge 2} (F_j\CCC^{\bt}_R(A)/F_{j-k-1}) \hbar^{j-1}$, for $F$  the good truncation filtration in the Hochschild direction. Via induction on levels of the filtration, and an analysis of the associated DGLA obstruction theory,
 Proposition \ref{quantprop} then shows  that the resulting map
\[
 Q\cP(A,0)^{\nondeg} \to (Q\cP(A,0)^{\nondeg}/G^2) \by \hbar^2\H^2(\DR(A))\llbracket \hbar\rrbracket
\]
underlies an equivalence. Thus  quantisation reduces to a first order problem.

This first order problem is resolved by introducing a notion of self-duality. In  \cite{DQvanish}, self-dual quantisations were defined for line bundles $\sL$ with an involution $\sL \simeq \sL^{\vee}$ to the  Grothendieck--Verdier dual. The analogous notion in our setting is given by considering anti-involutive associative algebras and categories.   Explicitly, when $X$ is a smooth variety, a self-dual quantisation of $\O_X$ is an associative deformation
\[
 (\O_X\llbracket \hbar \rrbracket, \star_{\hbar}) 
\]
of $\O_X$ with $a\star_{-\hbar} b= b\star_{\hbar} a$; the explicit quantisation formula of \cite{kontsevichPoisson} satisfies this property. More generally, a self-dual quantisation of $X$  over $R$ leads to a curved $A_{\infty}$-category with $R\brh$-semilinear  anti-involution, deforming the dg category of perfect complexes on $X$.

Restricting to self-dual quantisations ensures that the first-order obstruction vanishes, leading to Theorem \ref{quantpropsd}, which shows that the equivalence class of self-dual quantisations of a given non-degenerate Poisson structure is parametrised by
\[
 \hbar^2\H^2(\DR(A))\llbracket \hbar^2\rrbracket,
\]
and in particular such quantisations always exist. Global versions of these results for derived Artin $N$-stacks are summarised in Theorem \ref{Artinthm}.

The structure of the paper is as follows.

In \S \ref{defsn} we recall the description from \cite{poisson} of commutative bidifferential bigraded algebras
as formal completions of derived $N$-stacks along derived affines, together with the complex of polyvectors $\widehat{\Pol}(A,0)$ on such objects, and the space $\cP(A,0)$ of Poisson structures. We then introduce a quantisation  $Q\widehat{\Pol}(A,0)$ of the complex of polyvectors, defined in terms of the Hochschild complex, and introduce an anti-involution of this complex whose fixed points give rise to self-dual quantisations. 

\S \ref{compatsn} contains the technical heart of the paper. After introducing generalised pre-symplectic structures as de Rham power series, and after recalling formality quasi-isomorphisms for the $E_2$ operad associated to Levi decompositions $w$ of the Grothendieck--Teichm\"uller group, we introduce the notion (Definition \ref{Qcompatdef}) of $w$-compatibility between quantisations and generalised pre-symplectic structures. The main results (Propositions \ref{QcompatP1},  \ref{quantprop} and Theorem \ref{quantpropsd}) then follow, establishing the existence of quantisations of non-degenerate unshifted Poisson (and hence symplectic) structures on stacky derived affines. Proposition \ref{cfKT} shows that for  Levi decompositions corresponding to even associators, constant power series correspond to Kontsevich--Tamarkin quantisations. 

In  \S \ref{stacksn}, these results are translated into the fully global setting of derived Artin $N$-stacks (Theorem \ref{Artinthm}). The approach precisely mimics that of \cite[\S\S \ref{poisson-DMsn}, \ref{poisson-Artinsn}]{poisson}, by establishing \'etale functoriality in an $\infty$-category setting and applying descent arguments. Proposition \ref{Perprop2} shows how $E_1$ quantisations in our sense give rise to curved $A_{\infty}$ deformations of the dg category of perfect complexes on a derived Artin $N$-stack.

I would like to thank the anonymous referee for many helpful comments.

\tableofcontents

\section{
Quantisation for stacky thickenings of derived affine schemes}\label{defsn}


\subsection{Stacky thickenings of derived affines}
 
We now recall some definitions and lemmas from \cite[\S \ref{poisson-Artinsn}]{poisson}, as summarised in \cite[\S \ref{DQvanish-bicdgasn}]{DQvanish}. By default, we will  regard the  CDGAs in derived algebraic geometry   as chain complexes $\ldots \xra{\delta} A_1 \xra{\delta} A_0 \xra{\delta} \ldots$ rather than cochain complexes --- this will enable us to distinguish easily between derived (chain) and stacky (cochain) structures.  

\begin{definition}
A stacky CDGA is  a chain cochain complex $A^{\bt}_{\bt}$ equipped with a commutative product $A\ten A \to A$ and unit $\Q \to A$.  Given a chain  CDGA $R$, a stacky CDGA over $R$ is then a morphism $R \to A$ of stacky CDGAs. We write $DGdg\CAlg(R)$ for the category of  stacky CDGAs over $R$, and $DG^+dg\CAlg(R)$ for the full subcategory consisting of objects $A$ concentrated in non-negative cochain degrees.
\end{definition}
As explained in \cite[Remark \ref{poisson-cfCPTVV}]{poisson}, these correspond to the ``graded mixed cdgas'' of \cite{CPTVV} (but beware that the latter do not have mixed differentials).

When working with chain cochain complexes $V^{\bt}_{\bt}$, we will usually denote the chain differential by $\delta \co V^i_j \to V^i_{j-1}$, and the cochain differential by $\pd \co V^i_j \to V^{i+1}_j$.
Readers interested only in DM (as opposed to Artin) stacks may ignore the stacky part of the structure and consider only chain CDGAs  $A_{\bt}= A^0_{\bt}$ throughout this section.

\begin{example}\label{DstarBGex}
We now recall an important example  of a class of stacky CDGAs from \cite[Example \ref{poisson-DstarBG}]{poisson}.  
 Given a Lie algebra $\g$ of finite rank acting as derivations on a derived affine scheme $Y$, we write $O([Y/\g])$ for the stacky CDGA given by the
 Chevalley--Eilenberg double complex 
\[
  O(Y) \xra{\pd} O(Y)\ten \g^{\vee} \xra{\pd} O(Y)\ten \Lambda^2\g^{\vee}\xra{\pd} \ldots 
\]
of $\g$ with coefficients in the chain $\g$-module $O(Y)$.

When the action of $\g$ on $Y$ is induced by the action of an affine group scheme $G$ with Lie algebra $\g$, the stacky CDGA can recover the relative de Rham stack $[Y/\g]$ of $Y$ over $[Y/G]$; explicitly, the stack $[Y/\g]$ is the quotient of $Y$ by the action of the group sheaf $B \mapsto \exp(\g\ten \ker(B\to B^{\red} ))$. 
Then  \cite[Example \ref{poisson-DstarBG}]{poisson} gives a formally \'etale  simplicial resolution of $[Y/G]$ in terms of the functors $[Y\by G^n/\g^{n+1}]$. 
\end{example}

\begin{definition}
 Say that a morphism $U \to V$ of chain cochain complexes is a levelwise quasi-isomorphism if $U^i \to V^i$ is a quasi-isomorphism for all $i \in \Z$. Say that a morphism of stacky CDGAs is a levelwise quasi-isomorphism if the underlying morphism of chain cochain complexes is so.
\end{definition}

The following is \cite[Lemma \ref{poisson-bicdgamodel}]{poisson}:
\begin{lemma}\label{bicdgamodel}
There is a cofibrantly generated model structure on stacky CDGAs over $R$ in which fibrations are surjections and weak equivalences are levelwise quasi-isomorphisms. 
\end{lemma}

There is a denormalisation functor $D$ from non-negatively graded CDGAs to cosimplicial algebras, with 
 left adjoint $D^*$ as in \cite[Definition \ref{ddt1-nabla}]{ddt1}. 
Given a cosimplicial chain CDGA $A$, $D^*A$ is then a stacky CDGA in non-negative cochain degrees. By  \cite[Lemma \ref{poisson-Dstarlemma}]{poisson}, $D^*$ is a left Quillen functor from the Reedy model structure on cosimplicial chain CDGAs to the model structure of Lemma \ref{bicdgamodel}.

Since   $DA$ is a pro-nilpotent extension of $A^0$, when $\H_{<0}(A)=0$ we think of the simplicial hypersheaf  $\oR \Spec DA$ as a stacky derived thickening of the derived affine scheme $\oR \Spec A^0$.  

\begin{definition}
 Given a chain cochain complex $V$, define the cochain complex $\hat{\Tot} V \subset \Tot^{\Pi}V$ by
\[
(\hat{\Tot} V)^m := (\bigoplus_{i < 0} V^i_{i-m}) \oplus (\prod_{i\ge 0}   V^i_{i-m})
\]
with differential $\pd \pm \delta$. 
\end{definition}
The key property of the semi-infinite total complex $\hat{\Tot}$ is that it sends levelwise quasi-isomorphisms in the chain direction to quasi-isomorphisms; the same is not true in general of the sum and product total complexes $\Tot, \Tot^{\Pi}$, cf. \cite[\S 5.6]{W}.
The functor $\hat{\Tot}$ is referred to as Tate realisation in \cite{CPTVV}.

\begin{definition}
 Given a stacky CDGA $A$ and $A$-modules $M,N$ in chain cochain complexes, we define  internal $\Hom$s
$\cHom_A(M,N)$  by
\[
 \cHom_A(M,N)^i_j=  \Hom_{A^{\#}_{\#}}(M^{\#}_{\#},N^{\#[i]}_{\#[j]}),
\]
with differentials  $\pd f:= \pd_N \circ f \pm f \circ \pd_M$ and  $\delta f:= \delta_N \circ f \pm f \circ \delta_M$,
where $V^{\#}_{\#}$ denotes the bigraded vector space underlying a chain cochain complex $V$. 

We then define the  $\Hom$ complex $\hat{\HHom}_A(M,N)$ by
\[
 \hat{\HHom}_A(M,N):= \hat{\Tot} \cHom_A(M,N).
\]
\end{definition}
Note that  there is a multiplication $\hat{\HHom}_A(M,N)\ten \hat{\HHom}_A(N,P)\to \hat{\HHom}_A(M,P)$  (the same is not true for $\Tot^{\Pi} \cHom_A(M,N)$ in general).

Writing $\Omega^1_A:= \Omega^1_{A/R}$, we have:
\begin{definition}\label{hfetdef}
 A morphism  $A \to B$ in $DG^+dg\CAlg(R)$ is said to be  homotopy formally \'etale when the map
\[
 \{\Tot \sigma^{\le q} (\oL\Omega_{A}^1\ten_{A}^{\oL}B^0)\}_q \to \{\Tot \sigma^{\le q}(\oL\Omega_{B}^1\ten_B^{\oL}B^0)\}_q
\]
is a pro-quasi-isomorphism (i.e. an essentially levelwise quasi-isomorphism in the sense of \cite[\S 2.1]{isaksenStrict}), where $\sigma^{\le q}$ denotes the brutal cotruncation
\[
 (\sigma^{\le q}M)^i := \begin{cases} 
                         M^i & i \ge q, \\ 0 & i<q.
                        \end{cases}
\]
\end{definition}

Combining \cite[Proposition \ref{poisson-replaceprop}]{poisson} with  \cite[Theorem \ref{stacks2-bigthm} and Corollary \ref{stacks2-Dequivcor}]{stacks2}, every strongly quasi-compact derived Artin $N$-stack over $R$ can be resolved by a derived DM hypergroupoid (a form of homotopy formally \'etale cosimplicial diagram) in $DG^+dg\CAlg(R)$.

\subsection{Polyvectors}\label{polyvectorsn}

We now fix a chain CDGA $R$  over $\Q$.

\begin{assumption}\label{biCDGAprops}
 As in \cite[\S  \ref{poisson-bipoisssn}]{poisson}, we now assume that    $A \in DG^+dg\CAlg(R)$ has the following properties:
\begin{enumerate}
 \item for any cofibrant replacement $\tilde{A}\to A$ in the model structure of Lemma \ref{bicdgamodel}, the morphism $\Omega^1_{\tilde{A}}\to \Omega^1_{A}$ is a levelwise quasi-isomorphism,
\item  the $A^{\#}$-module $(\Omega^1_{A})^{\#}$ in   graded chain complexes is cofibrant (i.e. it has the left lifting property with respect to all surjections of $A^{\#}$-modules in graded chain complexes),
\item there exists $N$ for which the chain complexes $(\Omega^1_{A}\ten_AA^0)^i $ are acyclic for all $i >N$.
\end{enumerate}
\end{assumption}

Of particular interest for us is that these conditions are satisfied when  $A= D^*O(X)$ for derived Artin $N$-hypergroupoids $X$. The following is adapted from \cite[Definition \ref{poisson-bipoldef}]{poisson} along the lines of \cite[Definition \ref{DQvanish-poldef}]{DQvanish}, with the introduction of a dummy variable $\hbar$ of cohomological degree $0$ to assist comparison with quantisation constructions.

%
%

\begin{definition}\label{poldef}
Define the complex of $0$-shifted polyvector fields (or strictly speaking, multiderivations) on $A$ by
\[
 \widehat{\Pol}(A,0):= \prod_{p \ge 0}\hat{\HHom}_A(\Omega^p_{A},A)\hbar^{p-1}[-p]. 
\]
with graded-commutative  multiplication $(a,b)\mapsto ab$ on $ \hbar\widehat{\Pol}(A,0)$  following the usual conventions for symmetric powers.  

The Lie bracket on $\hat{\Hom}_A(\Omega^1_{A},A)$ then extends to give a bracket (the Schouten--Nijenhuis bracket)
\[
[-,-] \co \widehat{\Pol}(A,0)\by \widehat{\Pol}(A,0)\to \widehat{\Pol}(A,0)[-1],
\]
determined by the property that it is a bi-derivation with respect to the multiplication operation. 

Thus  $\hbar \widehat{\Pol}(A,0)$ has the natural structure of a $P_{2}$-algebra (i.e. a Gerstenhaber algebra), and  $\widehat{\Pol}(A,0)[1]$ is a differential graded Lie algebra (DGLA) over $R$.
\end{definition}

%

\begin{definition}\label{Fdef}
Define a decreasing filtration $F$ on  $\widehat{\Pol}(A,0)$ by 
\[
 F^i\widehat{\Pol}(A,0):= \prod_{j \ge i}\hat{\HHom}_A(\Omega^j_{A},A)\hbar^{j-1}[-j];
\]
this has the properties that $\widehat{\Pol}(A,0)= \Lim_i \widehat{\Pol}(A,0)/F^i$, with $[F^i,F^j] \subset F^{i+j-1}$, $ (\pd \pm \delta) F^i \subset F^i$, and $ (\hbar F^i) (\hbar F^j) \subset \hbar F^{i+j}$.
\end{definition}

Observe that this filtration makes $F^2\widehat{\Pol}(A,n)[1]$ into a pro-nilpotent DGLA.


\begin{definition}\label{Tpoldef0}
 Define the tangent space of polyvectors by 
\[
 T\widehat{\Pol}(A,0):= \widehat{\Pol}(A,0)\oplus \widehat{\Pol}(A,0)\hbar\eps,
\]
for $\eps$ of degree $0$ with $\eps^2=0$. Then $T\widehat{\Pol}(A,0)[1] $  is a DGLA with Lie bracket  given by $ [u+v\eps, x+y\eps]= [u,x]+ [u,y]\eps + [v,x]\eps$.
\end{definition}

%
%
%


\begin{definition}\label{Tpoldef}
Given a Maurer--Cartan element $\pi \in  \mc(F^2\widehat{\Pol}(A,0)[1]) $, define 
\[
 T_{\pi}\widehat{\Pol}(A,0):= \prod_{p \ge 0}\hat{\HHom}_A(\Omega^p_{A},A)\hbar^{p}[-p],
\]
with derivation $(\pd \pm \delta) + [\pi,-]$ (necessarily square-zero by the Maurer--Cartan conditions). 

The product on polyvectors makes this a  CDGA, and it inherits the filtration $F$ from $\hat{\Pol}$ (so, ignoring the differentials, we have $ F^iT_{\pi}\widehat{\Pol}(A,0) \cong \hbar F^i\hat{\Pol}(A,0)$). 

Given $\pi \in \mc(F^2\widehat{\Pol}(A,0)[1]/F^p)$, we define $T_{\pi}\widehat{\Pol}(A,0)/F^p$ similarly. This is a CDGA  because $F^i\cdot F^j \subset F^{i+j}$.
\end{definition}

Regarding $ T_{\pi}\widehat{\Pol}(A,0)[1]$ as an abelian DGLA, observe that $\mc(T_{\pi}\widehat{\Pol}(A,0)[1])$ is just the fibre of $\mc( T\widehat{\Pol}(A,0)[1]) \to  \mc(\widehat{\Pol}(A,0)[1])$ over $\pi$. 

%

\subsection{The Hochschild complex of a stacky CDGA}\label{HHsn}

\begin{definition}\label{HHdef}
For an $A$-module $M$ in chain cochain complexes, 
 we define the cohomological Hochschild complex $\CCC^{\bt}_R(A,M)$  over $R$ as we would for dg algebras, but using the $\Hom$-complexes $\hat{\HHom}$. Thus $\CCC^{\bt}_R(A,M)$ is the product total complex of a double complex $\uline{\CCC}^{\bt}_R(A,M)$ given by  
\[
 \uline{\CCC}^n_R(A,M)= \hat{\HHom}_R(A^{\ten_R n}, M),
\]
with Hochschild differential $b \co \uline{\CCC}^{n-1} \to \uline{\CCC}^{n}$ given by
\begin{align*}
 (b f)(a_1, \ldots , a_n) = &a_1 f(a_2, \ldots, a_n)\\
 &+ \sum_{i=1}^{n-1}(-1)^i f(a_1, \ldots, a_{i-1}, a_ia_{i+1}, a_{i+2}, \ldots, a_n)\\
&+ (-1)^n f(a_1, \ldots, a_{n-1})a_n.
\end{align*} 

There is also a quasi-isomorphic normalised version $N_c\uline{\CCC}^{\bt}_R(A,M)$, given by the subspaces of functions $f$ with $f(a_1, \ldots, a_{i-1}, 1,a_i,  \ldots, a_n)=0$ for all $i$. 

We define increasing filtrations $F$ on $\uline{\CCC}^{\bt}_R(A,M)$ and $\CCC^{\bt}_R(A,M)$ by good truncation in the Hochschild direction, so $F_p \CCC^{\bt}_R(A,M) \subset \CCC^{\bt}_R(A,M)$ is the subspace
\[
\prod_{i=0}^{p-1} \uline{\CCC}^i_R(A,M)[-i] \by \ker(b \co\uline{\CCC}^p_R(A,M) \to \uline{\CCC}^{p+1}_R(A,M))[-p].
\]

We simply write $\CCC^{\bt}_R(A)$ and $\uline{\CCC}^{\bt}_R(A)$ for $\CCC^{\bt}_R(A,A)$ and $\uline{\CCC}^{\bt}_R(A,A)$.
\end{definition}

\begin{lemma}\label{gradedHH}
 There is a natural brace algebra structure on $\CCC^{\bt}_R(A)$ over $R$, compatible with the filtration $F$. In particular, $\CCC^{\bt}_R(A)[1]$ is a filtered DGLA over $R$.  On the associated graded brace algebra $\gr^F\CCC^{\bt}_R(A) $, the Lie bracket and higher braces vanish, and there is a surjective quasi-isomorphism
\[
\gr^F\CCC^{\bt}_R(A) \to \hat{\Tot}(\HH^*_{R_{\#}}(A^{\#}_{\#}), \pd \pm \delta)
\]
of brace algebras, where we set all the braces to be $0$ on $\HH^*$.
\end{lemma}
\begin{proof}
 As in  \cite[\S 3]{voronovHtpyGerstenhaber}, there is a brace algebra structure on $\CCC^{\bt}_R(A)$, with cup product $\cdot$ of cohomological degree $0$ and brace operations $(f, g_1, \ldots , g_n) \mapsto \{f\}\{g_1, \ldots, g_n\}$ of cohomological degree $-n$. Writing $[f,g] := \{f\}\{g\} - (-1)^{\deg f \deg g} \{g\}\{f\}$ defines a Lie bracket of degree $-1$, making $ \CCC^n_R(A)[1]$ into a DGLA. 

Compatibility of $b$ with the bracket then implies that $[F_p, F_q] \subset F_{p+q-1}$, and degree considerations also give 
\[
 F_p \cdot F_q \subset F_{p+q}, \quad \{F_p\}\{F_{r_1}, \ldots, F_{r_n}\} \subset F_{p+ r+1-n},
\]
where $r= \sum r_i$. Since $F_{p+ r+1-n}\subset F_{p+r}$, this ensures that $(\CCC^{\bt}_R(A),F)$ is a filtered brace algebra; the bracket vanishes on $\gr^F$, as do the braces for $n \ge 2$. 

Since $F$ is defined as good truncation in the Hochschild direction, Hochschild cohomology $\HH^*$ is automatically a quasi-isomorphic quotient  of $\gr^F$. Any operation of negative degree necessarily vanishes on this quotient, so the quotient map is a brace algebra morphism. 
\end{proof}


\begin{lemma}\label{involutiveHH}
There is an involutive   map $i \co \CCC^{\bt}_R(A)[1] \to \CCC^{\bt}_R(A)[1]$ of DGLAs given by
\[
 i(f)(a_1, \ldots, a_m) = - (-1)^{\sum_{i<j}  \deg a_i \deg a_j} (-1)^{m(m+1)/2}f(a_m, \ldots , a_1).
\]

This involution corresponds under the HKR isomorphism to the involution of $\widehat{\Pol}(A,0)$ which acts on  $\hat{\HHom}_A(\Omega^p_{A},A)$ as scalar multiplication by $(-1)^{p-1}$.
\end{lemma}
\begin{proof}
The first statement is proved in  \cite[\S 2.1]{braunInvolutive}, taking the trivial involution on $A$.
For the second statement, given $\phi \in \hat{\HHom}_A(\Omega^p_{A},A)$, the corresponding element $f$ of $\CCC^p(A)$ is given by $f(a_1, \ldots, a_p):= \phi(da_1\wedge \ldots \wedge da_p)$, and then
\begin{align*}
i(f)(a_1, \ldots, a_p)&= - (-1)^{\sum_{i<j}  \deg a_i \deg a_j} (-1)^{p(p+1)/2}\phi( da_p\wedge \ldots \wedge da_1)\\
&= -(-1)^{p(p+1)/2}(-1)^{p(p-1)/2} \phi(da_1\wedge \ldots \wedge da_p)\\
&= -(-1)^{p} f(a_1, \ldots, a_p).
\end{align*}
\end{proof}

\subsection{Quantised $0$-shifted polyvectors and quantisations}

\begin{definition}\label{qpoldef}
Define the complex of quantised $0$-shifted polyvector fields  on $A$ by
\[
 Q\widehat{\Pol}(A,0):= \prod_{p \ge 0}F_p \CCC^{\bt}_R(A)\hbar^{p-1}. 
\]
\end{definition}

Properties of the filtration $F$ from Lemma \ref{gradedHH}  ensure that $Q\widehat{\Pol}(A,0)[1]$ is a DGLA.

\begin{definition}\label{QFdef}
Define a decreasing filtration $\tilde{F}$ on  $Q\widehat{\Pol}(A,0)$  by the subcomplexes 
\begin{align*}
 \tilde{F}^iQ\widehat{\Pol}(A,0)&:= \prod_{j \ge i} F_j\CCC^{\bt}_R(A)\hbar^{j-1}.
\end{align*}
 \end{definition}
This filtration is complete and Hausdorff,   with $[\tilde{F}^i,\tilde{F}^j] \subset \tilde{F}^{i+j-1}$.
In particular,  this makes $\tilde{F}^2Q\widehat{\Pol}(A,0)[1]$  into a pro-nilpotent filtered DGLA.

\begin{definition}
Define an $E_1$ quantisation of $A$ over $R$ to be a Maurer--Cartan element
\[
 \Delta \in \mc(\tilde{F}^2Q\widehat{\Pol}(A,0)).
\]
\end{definition}
When $A=A^0_{\bt}$ is just a CDGA,  this gives a curved $A_{\infty}$-algebra structure $A'$ on $A\llbracket\hbar\rrbracket$ with $A'/\hbar=A$,  because $\hbar\mid \Delta$. For more general stacky CDGAs, the stacky and derived structures interact in a non-trivial way for quantisations, and indeed for Poisson structures. 

\begin{remark}\label{BVrmk1}
 To strengthen the analogy between this construction and \cite{DQvanish}, we could replace $N_c\CCC^{\bt}(A) $ with its quasi-isomorphic  subcomplex of polydifferential operators. The filtration $F$ is then quasi-isomorphic to the order filtration for polydifferential operators, but the latter does not interact so well with the Lie bracket.

If we wished to consider uncurved $A_{\infty}$-algebra deformations without inner automorphisms, we would have to replace $\CCC^{\bt}(A)$ with its sub-DGLA
   $\ker( \CCC^{\bt}_R(A)\to \hat{\Tot}A)$. The analogue for  \cite{DQvanish} is the kernel of the map $\sD_A \to A$ given by evaluating at $1$.
As in  \cite[Remark \ref{DQvanish-BVrmk}]{DQvanish}, this means that the $E_0$ analogue of a strict quantisation is a  BV algebra deformation.
\end{remark}

\begin{example}\label{quantex}
 When the stacky CDGA $A$ is bounded in the stacky (cochain) direction, we may identify $\CCC^{\bt}_R(A)$ with the Hochschild complex of the CDGA $\Tot A$, as $\cHom(A^{\ten n},A)$ is then  also bounded in the cochain direction, and the functors $\Tot, \hat{\Tot}, \Tot^{\Pi}$ agree for such double complexes. In particular, this applies to stacky CDGAs of the form $O([Y/\g])$  in the notation of Example \ref{DstarBGex}. 

Given a finite rank Lie algebra $\g$ acting on a smooth affine $Y$ over $R$, the derived cotangent stack $T^*[Y/\g]$ carries a non-degenerate Poisson structure. Explicitly, if $Y=\Spec B$, this derived formal stack is represented by the stacky CDGA given by the Chevalley--Eilenberg complex $O( T^*[Y/\g]):=O([\Spec \Symm_B(\cone(\g \ten_R B \to T_{B}))/\g])$, and we then have $\Tot O( T^*[Y/\g])= \Symm_{ O([Y/\g])}T_{O([Y/\g])}$ with its natural Poisson structure as a complex of polyvectors.

A quantisation of this Poisson structure is  given by the Rees algebra $\prod_i \hbar^iF_i\cD_{\Tot O([Y/\g])}$ of the order filtration $F$ on the ring of differential operators, i.e.  the $\hbar$-adically complete sub-DGA of $\cD_{\Tot O([Y/\g])}\brh$ generated by $ O([Y/\g])$ and first order differential operators divisible by $\hbar$. This quantisation satisfies $b\star_{\hbar}a =(-1)^{\deg a \deg b} a\star_{-\hbar}b$, so will be included in the parametrisation of  Theorem \ref{quantpropsd}.
\end{example}

\begin{definition}\label{mcPLdef}
 Given a   DGLA $L$, define the the Maurer--Cartan set by 
\[
\mc(L):= \{\omega \in  L^{1}\ \,|\, d\omega + \half[\omega,\omega]=0 \in   L^{2}\}.
\]

Following \cite{hinstack}, define the Maurer--Cartan space $\mmc(L)$ (a simplicial set) of a nilpotent  DGLA $L$ by
\[
 \mmc(L)_n:= \mc(L\ten_{\Q} \Omega^{\bt}(\Delta^n)),
\]
where 
\[
\Omega^{\bt}(\Delta^n)=\Q[t_0, t_1, \ldots, t_n,\delta t_0, \delta t_1, \ldots, \delta t_n ]/(\sum t_i -1, \sum \delta t_i)
\]
is the commutative dg algebra of de Rham polynomial forms on the $n$-simplex, with the $t_i$ of degree $0$.
\end{definition}

\begin{definition}\label{Gdef}
We now define another decreasing filtration $G$ on  $Q\widehat{\Pol}(A,0)$ by setting
\begin{align*}
 G^iQ\widehat{\Pol}(A,0)&:= \hbar^{i}Q\widehat{\Pol}(A,0).
\end{align*}
We then set $G^i \tilde{F}^p:= G^i \cap \tilde{F}^p$.
\end{definition}

\begin{definition}\label{Qpoissdef}
Define  the space $Q\cP(A,0)$ of $E_1$ quantisations of $A$ over $R$  to be given by the simplicial 
set
\[
 Q\cP(A,0):= \Lim_i \mmc( \tilde{F}^2 Q\widehat{\Pol}(A,0)[1]/\tilde{F}^{i+2}).
\]
Also write
\[
 Q\cP(A,0)/G^k:= \Lim_i\mmc(\tilde{F}^2 Q\widehat{\Pol}(A,0)[1]/(\tilde{F}^{i+2}+G^k)),
\]
 so $Q\cP(A,0)= \Lim_k Q\cP(A,0)/G^k$. 
\end{definition}

When $R$ and $A=A^0$ are concentrated in non-negative  homological degrees, we can  interpret $Q\cP(A,0)$ as a space of deformations of $A$ as an $R$-linear dg category up to quasi-equivalence, and in general when $A=A^0$ and has bounded cohomology,  \cite{LowenvdBerghCurvature,BlancKatzarkovPandit} interpret $Q\cP(A,0)$  as a space of deformations of $A$ as an $R$-linear dg category up to derived Morita equivalence.

Since the functor $\hat{\Tot}$ is lax monoidal with respect to tensor products, 
for stacky CDGAs we have a natural map $\CCC^{\bt}_R(A)\to\CCC^{\bt}_R(\hat{\Tot}A)$ (rarely an equivalence), so $E_1$ quantisations give rise to curved $A_{\infty}$ deformations of the CDGA $\hat{\Tot}A$. We now give a stronger statement.

\begin{definition}\label{Perdef}
 If $A  \in DG^+dg\CAlg(R)$, define the bi-dg category $\cPer(A)$ as follows. Objects are $A$-modules $M$ in chain cochain complexes for which $M^{\#}$ is cofibrant as a graded chain complex over $A^{\#}$, $M^0$ is perfect over $A^0$, and the map $M^0\ten_{A^0}A^{\#} \to M^{\#}$ is a levelwise quasi-isomorphism. Morphisms are given by the chain cochain complexes $\cHom_A(M,N)$.

We then define $\per_{dg}(A)$ to have the same objects as $\cPer(A)$, and morphisms $\hat{\HHom}_A(M,N)$.
\end{definition}
Note that the $\infty$-category underlying $\per_{dg}(A)$ is the category of perfect modules featuring in \cite[Proposition 2.2.8]{CPTVV}. 

For every $M \in \per_{dg}(A)$, we have a $\hat{\Tot}A$-module $\hat{\Tot}M$, but this need not be cofibrant or perfect. For instance, given $b \in \z_0\z^1A$, we may set $A_b$ to be the chain cochain complex $A^{\#}_{\bt}$ with cochain differential $\pd_A+ b$. Since $A_b^{\#}=A^{\#}$, it lies in $\cPer(A)$, but $\hat{\Tot}A_b$ is seldom cofibrant.

\begin{proposition}\label{Perprop}
For $A  \in DG^+dg\CAlg(R)$,  there is a natural map in the $\infty$-category of simplicial sets from $Q\cP(A,0)$ to the space of curved $A_{\infty}$ deformations $(\per_{dg}(A)\llbracket \hbar \rrbracket, \{m^{(i)}\}_{i\ge 0})$ of the dg category $\per_{dg}(A) $, with $ \hbar^{i-1}\mid m^{(i)}$ for $i \ge 3$.
\end{proposition}
\begin{proof}
For any $R$-linear bi-dg category $\cB$, we have a Hochschild complex built from the spaces 
\[
 \uline{\CCC}^n_R(\cB)= \prod_{x_0, \ldots x_n \in \cB}\hat{\HHom}_R(\cB(x_0,x_1)\ten_R\ldots\ten_R \cB(x_{n-1},x_n), \cB(x_0, x_n)),
\]
with $Q\cP(\cB,0)$ defined analogously. Properties of $\hat{\Tot}$ then give us a natural map from  $Q\cP(\cB,0)$ to $Q\cP(\hat{\Tot}\cB,0) $, which is the space of curved $A_{\infty}$ deformations $( (\hat{\Tot}\cB)\llbracket \hbar \rrbracket, \{m^{(i)}\}_{i\ge 0})$ of the dg category $\hat{\Tot}\cB $ with $ \hbar^{i-1}\mid m^{(i)}$ for $i \ge 3$; the Maurer--Cartan conditions ensure that $\hbar^i\mid bm^{(i)}$, so every such $m$ does lie in the appropriate piece of the good truncation filtration.  

It therefore suffices to show that the map $Q\cP(\cPer(A),0)\to Q\cP(A,0)$ given by restriction to the object $A\in\cPer(A) $ is a weak equivalence. By the theory of pro-nilpotent DGLAs, this will follow if $ \uline{\CCC}^n_R(\cPer(A))\to \uline{\CCC}^n_R(A)$ is a filtered quasi-isomorphism. 


We now  observe that for any $A$-linear bi-dg category $\cB$ with cofibrant $\cHom$-bicomplexes,  there is a spectral sequence  
\[
 \HH^i_{R_{\#}}(A^{\#}_{\#}, \HH^j_{A^{\#}_{\#}}(\cB^{\#}_{\#})) \abuts \HH^{i+j}_{R_{\#}}(\cB^{\#}_{\#}).
\]
When $\cB$ is homotopy Cartesian in the sense that the map $\cB^0\ten_{A^0}A^{\#} \to \cB^{\#}$ is a levelwise quasi-isomorphism, we have a quasi-isomorphism $ (\HH^j_{A^0_{\#}}(\cB^0_{\#}, (\cB^0\ten_{A^0}^{\oL}A^{\#})_{\#}), \delta)\simeq (\HH^j_{A^{\#}_{\#}}(\cB^{\#}_{\#}), \delta) $. We then note that 
 when $\cB^0_{\#}$ is  Morita equivalent 
to $A^0_{\#}$ as a graded category, the map $M \to \HH^*_{A^0_{\#}}(\cB^0_{\#},(\cB^0_{\#}\ten_{A^0_{\#}}^{\oL}M))$ is an isomorphism of graded modules for all $A^0_{\#}$-modules $M$. 

Putting these together gives quasi-isomorphisms
\[
 \gr^F_j\CCC^{\bt}_R(A)\simeq \gr^F_j\CCC^{\bt}_R(\cPer(A)):
\]
  the bi-dg category  $\cPer(A)$ is homotopy Cartesian because its objects are; since $\cPer(A)^0_{\#}$ is equivalent to the category of graded projective $A^0_{\#}$-modules, 
it is  Morita equivalent to $A^0_{\#}$. Thus $Q\cP(\cPer(A),0)\to Q\cP(A,0)$ is indeed a weak equivalence.
\end{proof}

\begin{remark}\label{gerbermk}
In \cite{DQvanish}, we were able to consider  $E_0$ quantisations not just of the structure sheaf $\O_X$, but also of line bundles, by 
constructing a $\bG_m$-action on  quantised polyvectors. 
  
Similarly, the methods of this paper can be adapted to study $E_1$ quantisations of any $A$-linear bi-dg category $\cB$ for which the map $\hat{\Tot} A \to  \gr^F\CCC^{\bt}_A(\cB) $ is a quasi-isomorphism --- by analogy, line bundles are $A$-modules for which the map $\hat{\Tot}A \to \oR\hat{\HHom}_A(M,M)$ is a quasi-isomorphism. In particular, we can study  \'etale $\bG_m$-gerbes by establishing $B\bG_m$-equivariance. 
  One way to do this is to consider $Q\cP(\cPer(A),0)$ as in the proof of Proposition \ref{Perprop}, since  $\cPer(A)$  admits   an action of the Picard $2$-group and hence a $B\bG_m$-action. 


The resulting action is necessarily trivial modulo $G^1$, so comes from  pro-unipotent $L_{\infty}$-automorphisms of  $Q\widehat{\Pol}(A,0)$. 
 Since pro-unipotent $L_{\infty}$-automorphisms are exponentials of pro-nilpotent  $L_{\infty}$-derivations, we will in fact have an action of $B\bG_m\ten_{\Z}\Q$, so a notion of quantisation for $(\bG_m\ten_{\Z}\Q)$-gerbes.
\end{remark}

\subsection{The centre of a quantisation}\label{centresn}


\begin{definition}\label{TQpoldef0}
 Define the filtered tangent space to quantised polyvectors by 
\begin{align*}
 TQ\widehat{\Pol}(A,0)&:= Q\widehat{\Pol}(A,0)\oplus \prod_{p \ge 0}F_p\CCC^{\bt}_R(A)\hbar^{p}\eps,\\
\tilde{F}^jTQ\widehat{\Pol}(A,0)&:= \tilde{F}^jQ\widehat{\Pol}(A,0)\oplus \prod_{p \ge j}F_p\CCC^{\bt}_R(A)\hbar^{p}\eps,
\end{align*}
for $\eps$ of degree $0$ with $\eps^2=0$. Then  $TQ\widehat{\Pol}(A,0)[1] $ is a DGLA, with  Lie bracket  given by $ [u+v\eps, x+y\eps]= [u,x]+ [u,y]\eps + [v,x]\eps$. 
\end{definition}

\begin{definition}\label{TQpoldef}
Given a Maurer--Cartan element $\Delta \in  \mc(\tilde{F}^2Q\widehat{\Pol}(A,0)[1]) $, define the centre of $(A,\Delta)$ by
\[
 T_{\Delta}Q\widehat{\Pol}(A,0):= \prod_{p \ge 0}F_p\CCC^{\bt}_R(A)\hbar^{p},
\]
with derivation $\pd \pm \delta \pm b + [\Delta,-]$ (necessarily square-zero by the Maurer--Cartan conditions).
 
This has a filtration
\[
 \tilde{F}^iT_{\Delta}Q\widehat{\Pol}(A,0):= \prod_{p \ge i}F_p \CCC^{\bt}_R(A)\hbar^{p},
\]
making $T_{\Delta}Q\widehat{\Pol}(A,0)$ a filtered brace algebra by Lemma \ref{gradedHH}.
  Given $\Delta \in  \mc(F^2Q\widehat{\Pol}(A,0)/\tilde{F}^p)$, we define $T_{\Delta}Q\widehat{\Pol}(A,0)/\tilde{F}^p$ similarly ---  this is also a brace algebra as $\tilde{F}^p$ is a brace ideal.
\end{definition}

Observe that $T_{\Delta}Q\cP(A,0):= \mmc(\tilde{F}^2T_{\Delta}Q\widehat{\Pol}(A,0)[1])$ is just the fibre of $\mmc( \tilde{F}^2TQ\widehat{\Pol}(A,0)[1]) \to  \mmc(\tilde{F}^2 Q\widehat{\Pol}(A,0)[1])$ over $\Delta$. 

Similarly to Definition \ref{Gdef}, there are filtrations $G$ on  
$TQ\widehat{\Pol}(A,0), T_{\Delta}Q\widehat{\Pol}(A,0) $ given by powers of $\hbar$. 
Since $\gr_G^i\tilde{F}^{p-i}Q\widehat{\Pol}= \prod_{j \ge p-i} \gr^F_{j-i}\CCC^{\bt}_R(A)\hbar^{j-1}$,
the 
HKR isomorphism gives maps 
  \begin{align*}
 \gr_G^i\tilde{F}^pQ\widehat{\Pol}(A,0) &\to 
 \prod_{j \ge p} \hat{\HHom}_A(\Omega^{j-i}_{A},A)\hbar^{j-1}[i-j]\\
\gr_G^i\tilde{F}^pT_{\Delta}Q\widehat{\Pol}(A,0) &\to \prod_{j \ge p}\hat{\HHom}_A(\Omega^{j-i}_{A},A) \hbar^{j}[i-j],
\end{align*}
which are quasi-isomorphisms by our hypotheses on $A$ (see Assumption \ref{biCDGAprops}).

For the filtration $F$ of Definition \ref{Fdef}, we may rewrite these maps as
 \begin{align*}
 \gr_G^i\tilde{F}^pQ\widehat{\Pol}(A,0) &\to  F^{p-i}\widehat{\Pol}(A,0)\hbar^{i},\\ 
\gr_G^i\tilde{F}^pT_{\Delta}Q\widehat{\Pol}(A,0) &\to F^{p-i}T_{\pi_{\Delta}}\widehat{\Pol}(A,0)\hbar^{i},
\end{align*}
where $\pi_{\Delta} \in \mc(F^2\widehat{\Pol}(A,0)[1])$ denotes the image of $\Delta$ under the map  $ \gr_G^0\tilde{F}^2Q\widehat{\Pol}(A,0) \to  F^2\widehat{\Pol}(A,0)$. 

Since the cohomology groups of $T_{\pi_{\Delta}}\widehat{\Pol}(A,0)$ are Poisson  cohomology, we will refer to the cohomology groups of  $T_{\Delta}Q\widehat{\Pol}(A,0)$ as quantised Poisson cohomology.

\begin{definition}\label{Qnonnegdef}
 Say that an $E_1$ quantisation $\Delta =  \sum_{j \ge 2} \Delta_j \hbar^{j-1}$ is non-degenerate if the map 
\[
\Delta_2^{\sharp}\co  \Tot^{\Pi} (\Omega_{A}^1\ten_AA^0) \to \hat{\HHom}_A(\Omega^1_A, A^0)
\]
is a quasi-isomorphism and $\Tot^{\Pi} (\Omega_{A}^1\ten_AA^0)$ is a perfect complex over $A^0$.
\end{definition}

\begin{definition}\label{TQPdef}
Define the tangent spaces
\begin{eqnarray*}
 TQ\cP(A,0)&:=& \Lim_i \mmc( \tilde{F}^2 TQ\widehat{\Pol}(A,0)[1]/\tilde{F}^{i+2}),
\end{eqnarray*}
with $TQ\cP(A,0)/G^k$,  defined similarly.
 \end{definition}
These are simplicial sets over $Q\cP(A,0)$ (resp.   $Q\cP(A,0)/G^k$), fibred in simplicial abelian groups. 

\begin{definition}\label{Qsigmadef}
Define the canonical tangent vector
\[
  \sigma=-\pd_{\hbar^{-1}}\co Q\widehat{\Pol}(A,0) \to TQ\widehat{\Pol}(A,0)
\]
 by $\alpha \mapsto \alpha + \eps \hbar^{2}\frac{\pd \alpha}{\pd \hbar}$. Note that this is a morphism of filtered  DGLAs, so gives a map $ \sigma \co Q\cP(A,0)\to TQ\cP(A,0)$, with $\sigma(\Delta) \in \z^2(\tilde{F}^2T_{\Delta}Q\widehat{\Pol}(A,0))$. 
\end{definition}

\subsection{Self-dual quantisations}\label{sdsn}

\begin{definition}
 Define an involution $Q\widehat{\Pol}(A,0) \xra{*} Q\widehat{\Pol}(A,0)$ by $\Delta^*(\hbar):= i(\Delta)(-\hbar)$, for the involution  $i$ of Lemma \ref{involutiveHH}. 
\end{definition}

\begin{definition}
 Lemma \ref{involutiveHH} ensures that  $*$ is a morphism of DGLAs, and we define the space $ Q\cP(A,0)^{sd} \subset   Q\cP(A,0)$  of self-dual quantisations to be the fixed points of the involution $*$. This inherits cofiltrations $\tilde{F}$ and $G$ from $Q\cP(A,0)$.
\end{definition}

In particular, this means that when $A=A^0$ is  concentrated in  degree zero, elements of $Q\cP(A,0)^{sd}$ can be represented  by associative algebra deformations $(A',\star_{\hbar})$ of $A$, with
\[
 a\star_{-\hbar} b = b\star_{\hbar} a. 
\]
More generally, when $R$ and $A=A^0$ are concentrated in non-negative  homological degrees, elements of  $Q\cP(A,0)^{sd}$ are algebroid quantisations of $A$ equipped with an anti-involution which is semilinear under the transformation $\hbar \mapsto -\hbar$. 

For general stacky CDGAs $A$, the trivial anti-involution on $A$ extends to the anti-involution $\cHom_A(-,A)$ on the dg category $\per_{\dg}(A)$ of perfect complexes,
and similarly to \cite{braunInvolutive}, this allows us to extend the involution of Lemma \ref{involutiveHH} to the Hochschild complex of $\per_{\dg}(A)$. Applying this to the constructions of  Proposition \ref{Perprop},
 $Q\cP(A,0)^{sd}$ gives rise  to  curved $A_{\infty}$-deformations $\tilde{\cP}_{\hbar}$ of $\per_{\dg}(A)$, equipped with an anti-involution $ \tilde{\cP}_{-\hbar} \simeq  \tilde{\cP}_{\hbar}^{\opp}$ lifting the duality functor $\cHom_A(-,A)$. 

\begin{remark}\label{sdgerbermk}
As in Remark \ref{gerbermk}, we may also consider self-duality for $\bG_m$-gerbes. Since the functor sending a gerbe to its opposite  is just given by the inversion map on $B^2\bG_m$, anti-involutive gerbes are  are classified by $B^2\mu_2$,  the homotopy fixed points of the inversion map.

However, as observed in Remark \ref{gerbermk}, the space of quantisations over $B^2\bG_m$ is the pullback of a space over $B^2(\bG_m\ten_{\Z}\Q)$, so  the space of self-dual quantisations over $B^2\mu_2$ is constant.  This means that to every self-dual quantisation of $A$ there correspond  self-dual quantisations of all $\mu_2$-gerbes, and in particular of $\per_{\dg}(A)$  with duality functor $\oR\hom(-,\sL)$ for any line bundle $\sL$. One way to make sense of this example is that even if $\sL$ does not have a square root, there is necessarily an automorphism of the Hochschild complex acting as a square root of $\sL$, and thus intertwining between the respective duality functors.  
\end{remark}
 
\begin{lemma}\label{filtsd}
There are canonical weak equivalences
 \begin{align*}
 Q\cP(A,0)^{sd}/G^{2i} &\to Q\cP(A,0)^{sd}/G^{2i-1}\\
Q\cP(A,0)^{sd}/G^{2i+1} &\to (Q\cP(A,0)^{sd}/G^{2i})\by^h_{(Q\cP(A,0)/G^{2i})}(Q\cP(A,0)/G^{2i+1}).
\end{align*}
\end{lemma}
\begin{proof}
This follows in much the same way as \cite[Lemma \ref{DQvanish-quantpropsd}]{DQvanish}. Lemma \ref{involutiveHH} ensures that the involution $*$ acts trivially on  $\widehat{\Pol}(A,0)$, since it maps $f \hbar^{p-1}$ to $ (-1)^{p-1}(-\hbar)^{p-1}$ for $f\in  \hat{\HHom}_A(\Omega^p_{A},A)\hbar^{p-1}$.
It therefore acts as multiplication by $(-1)^k$ on   $ \gr_G^k\widehat{\Pol}(A,0)= \hbar^k\gr_G^0\widehat{\Pol}(A,0)$, giving
quasi-isomorphisms
\[
\gr_G^k  \tilde{F}^pQ\widehat{\Pol}(A,0)^{sd} \simeq \begin{cases}
                                                         \gr_G^k  \tilde{F}^pQ\widehat{\Pol}(A,0) & k \text{ even}\\
0 & k \text{ odd}.
                                                        \end{cases}
\]
  The results then follow from the fibre sequences
\[
 Q\cP(A,0)^{sd}/G^{k+1} \to Q\cP(A,0)^{sd}/G^{k} \to \mmc(\gr_G^k  \tilde{F}^2Q\widehat{\Pol}(A,0)^{sd}[2]) 
\]
coming from  obstruction theory for abelian extensions of DGLAs.
\end{proof}

In particular, Lemma \ref{filtsd} gives  $Q\cP(A,0)^{sd}/G^{2} \simeq Q\cP(A,0)^{sd}/G^{1} \simeq \cP(A,0)$,  so every unshifted Poisson structure admits an essentially unique  first-order self-dual quantisation. 

\section{Quantisations and de Rham power series}\label{compatsn}
Recall that we are fixing  a chain
CDGA  $R$ over $\Q$, and a cofibrant stacky CDGA $A$ over $R$. We denote the chain differentials on $A$ and $R$  by $\delta$, and the  cochain differential on $A$ by $\pd$.

\subsection{Generalised pre-symplectic structures}

We now adapt some definitions from \cite[\S \ref{DQvanish-DRsn}]{DQvanish} and \cite[\S \ref{poisson-biprespsn}]{poisson}.

\begin{definition}\label{biDRdef}
Define the de Rham complex $\DR(A)$ to be the product total complex of the bicomplex
\[
 \Tot^{\Pi} A \xra{d} \Tot^{\Pi}\Omega^1_{A} \xra{d} \Tot^{\Pi}\Omega^2_{A}\xra{d} \ldots,
\]
so the total differential is $d \pm \pd \pm\delta $.

We define the Hodge filtration $F$ on  $\DR(A)$ by setting $F^p\DR(A) \subset \DR(A)$ to consist of terms $\Tot^{\Pi}\Omega^i_{A}$ with $i \ge p$. In particular, $F^p\DR(A)= \DR(A)$ for $p \le 0$.
\end{definition}

\begin{definition}
When $A$ is a cofibrant stacky CDGA over $R$, recall that  a $0$-shifted pre-symplectic structure $\omega$ on $A/R$ is an element
\[
 \omega \in \z^{2}F^2\DR(A).
\]
It is called symplectic if $\omega_2 \in \z^2\Tot^{\Pi}\Omega^2_{A}$ induces a quasi-isomorphism
\[
 \omega_2^{\sharp} \co \hat{\HHom}_A(\Omega^1_{A}, A^0)\to  \Tot^{\Pi} (\Omega_{A}^1\ten_AA^0) 
\]
and  $\Tot^{\Pi} (\Omega_{A}^1\ten_AA^0)$ is a perfect complex over $A^0$.
\end{definition}

\begin{definition}\label{tildeFDRdef}
Define a decreasing filtration $\tilde{F}$ on $ \DR(A)\llbracket\hbar\rrbracket$ by 
\[
 \tilde{F}^p\DR(A):= \prod_{i\ge 0} F^{p-i}\DR(A)\hbar^{i}.
\]

Define a further filtration $G$ by $ G^k \DR(A)\llbracket\hbar\rrbracket = \hbar^{k}\DR(A)\llbracket\hbar\rrbracket$.
\end{definition}

\begin{definition}\label{GPreSpdef}
 Define the space of generalised $0$-shifted pre-symplectic structures on $A/R$ to be the simplicial set
\[
 G\PreSp(A,0):= \Lim_i\mmc( \tilde{F}^2\DR(A)\llbracket\hbar\rrbracket[1]/\tilde{F}^{i+2}), 
\]
where we regard the cochain complex  $\DR(A)[1]$ as a  DGLA with trivial bracket. Write $\PreSp = G\PreSp/G^1$.

Also write $G\PreSp(A,0)/\hbar^{k}:= \Lim_i\mmc( (\tilde{F}^2\DR(A)[\hbar][1]/(G^k +\tilde{F}^{i+2)} )$, so $ G\PreSp(A,0)= \Lim_k G\PreSp(A,0)/\hbar^{k} $.

Set $G\Sp(A,0) \subset G\PreSp(A,0)$ to consist of the points whose images in $\PreSp(A,0)$ are  symplectic structures --- this is a union of path-components.
\end{definition}

\begin{remarks}
 Note that Definition \ref{GPreSpdef} is not the obvious analogue of the definition of generalised $(-1)$-shifted pre-symplectic structures from \cite[Definition \ref{DQvanish-GPreSpdef}]{DQvanish}, which used the convolution $(G*\tilde{F})^2= \tilde{F}^2+G^1$ in place of $\tilde{F}^2$ for reasons specific to negatively shifted structures. The only difference lies in the linear term, which is where the correspondence between generalised symplectic structures and non-degenerate quantisations breaks down anyway --- replacing $\tilde{F}^2$ with $(G*\tilde{F})^2$   would not significantly affect the main results of either paper, nor would eliminating the linear term altogether.

Also note that  $G\PreSp(A,0)$  is canonically weakly  equivalent to the Dold--Kan denormalisation of the good truncation complex $\tau^{\le 0}(\tilde{F}^2\DR(A)\llbracket\hbar\rrbracket[2])$ (and similarly for the various quotients we consider),   but the description in terms of $\mmc$ will simplify comparisons. In particular, we have
\[
 \pi_iG\PreSp(A,0)\cong \H^{2-i}(F^2\DR(A)) \by \hbar\H^{2-i}(F^1\DR(A))  \by \hbar^2\H^{2-i}(\DR(A))\llbracket\hbar\rrbracket. 
\]
\end{remarks}

\subsection{Formality}\label{formalitysn}

\begin{definition}
Write $\GT$ for the Grothendieck--Teichm\"uller group. This is an affine group scheme over $\Q$, with reductive quotient $\bG_m$. Denote the  pro-unipotent radical $\ker(\GT \to \bG_m)$ by $\GT^1$. 

Write $\Levi_{\GT}$ for the space of Levi decompositions of $\GT$, i.e. sections of $\GT \to \bG_m$. By the general theory  of pro-algebraic groups in characteristic $0$ (cf. \cite{humphreysLevi}, \cite[Theorem 3.2]{Levi}, or for instance \cite[Corollary 2.14]{htpy} in general), 
the space $\Levi_{\GT}$ is an affine scheme over $\Q$ equipped with the structure of a trivial $\GT^1$-torsor via the adjoint action, since the $\bG_m$-invariant subgroup of $\GT^1$ is trivial.
\end{definition}

Drinfeld associators \cite{drinfeldQuasitriangular,barnatan} form an affine $\Q$-scheme   $\Ass$ fibred over $\bG_m$. It is a bitorsor for $\GT$ (acting on the right) and the graded Grothendieck--Teichm\"uller group $\mathrm{GRT}$ acting on the left. Since $\mathrm{GRT}$ contains a distinguished copy of $\bG_m$, each  element $\Phi$ of $\Ass$ gives rise to a Levi decomposition $\sigma_{\Phi}\co \bG_m \to \GT$ characterised by the formula $\lambda \cdot \Phi = \Phi \cdot \sigma_{\Phi}(\lambda)$. We thus have an  isomorphism $\sigma_{?}\co \bG_m \backslash \Ass \to \Levi_{\GT}$, or equivalently $\Ass^1 \to \Levi_{\GT}$, of $\GT^1$-torsors. 

As explained succinctly in \cite{petersenGTformality}, formality of the $\Q$-linear  $E_2$ operad is a consequence of the observation that the   Grothendieck--Teichm\"uller group  is a pro-unipotent extension of $\bG_m$.  Since $\GT$ acts on $E_2$,  any Levi decomposition $w \co \bG_m \to \GT$ gives a weight decomposition (i.e. a  $\bG_m$-action) of $E_2$ which splits the good truncation filtration,  so gives an equivalence between $E_2$ and $P_2$. Since the natural morphism from the Lie operad to the  $E_2$ operad is given in each arity by inclusion of the top weight term for the decreasing filtration, it follows that such an equivalence $E_2 \simeq P_2$  automatically respects the natural maps from the Lie operad on each side. 

\begin{definition}
 Given a  Levi decomposition  $w \in \Levi_{\GT}(\Q)$, we denote by $p_w$ the resulting $\infty$-functor  from $E_2$-algebras to $P_2$-algebras over $\Q$, which respects the underlying $L_{\infty}$-algebras.
\end{definition}

 As in \cite{voronovHtpyGerstenhaber}, brace algebras are naturally $E_2$-algebras, so $\CCC^{\bt}_R(A)$ has an $E_2$-algebra structure. Moreover, the equivalence between $E_2$ and $P_2$ necessarily respects the good truncation filtrations, and the filtered complex $(\CCC^{\bt}_R(A),F)$ is an algebra with respect to the brace operad filtered by good truncation. This yields a filtered $P_2$-algebra $(p_{w}\CCC^{\bt}_R(A),F)$ over $A$ with $F_p \cdot F_q \subset F_{p+q}$ and $[F_p,F_q] \subset F_{p+q-1}$, with a filtered  $L_{\infty}$-quasi-isomorphism $ (p_{w}\CCC^{\bt}_R(A)[1],F)\simeq (\CCC^{\bt}_R(A)[1],F)$. 

\begin{definition}
 For any of the definitions from \S \ref{defsn}, we add the subscript $w$ to indicate that we are replacing $(\CCC^{\bt}_{R}(A),F) $ with $(p_{w}\CCC^{\bt}_{R}(A),F)$ in the construction. 
\end{definition}

Since these DGLAs are quasi-isomorphic and $\mmc$ preserves weak equivalences, in particular we have canonical weak equivalences $Q\cP_w(A,0) \simeq Q\cP(A,0)$. Properties of the filtration $\tilde{F}$ then ensure that the complexes $T_{\Delta}Q\widehat{\Pol}_w(A,0)$ are filtered $P_2$-algebras.

\begin{remark}\label{Levirmk}
 Rather than just choosing $w\in \Levi_{\GT}(\Q)$, a more natural approach might be to consider the simplicial set $\oR \Levi_{\GT}(R)$ of all Levi decompositions over $R$. This would lead to a space $Q\cP_{\Levi}(A,0)$ over $ \oR \Levi_{\GT}(R)$ with fibre $Q\cP_w(A,0)$ over $w$ and a canonical weak equivalence $Q\cP_{\Levi}(A,0) \simeq \oR \Levi_{\GT}(R) \by  Q\cP(A,0)$.
\end{remark}

\subsection{Compatible quantisations}

We will now develop the notion of compatibility between a generalised pre-symplectic structure and an $E_1$  quantisation, generalising the notion of compatibility between $0$-shifted pre-symplectic and Poisson structures from \cite{poisson}. The following definitions are adapted from \cite[Definition \ref{poisson-mudef}]{poisson}.

\begin{definition}\label{mudef}
Given a stacky CDGA $B$ over $A$ and a  derivation $\Delta \in \mc( \hat{\HHom}_B(\Omega^1_{B}, B))$, define 
\[
 \mu(-,\Delta) \co \DR(A) \to \hat{\Tot} B
\]
to be the morphism of graded $A$-algebras given on generators $ \Omega^1_{A}$ by setting
\[
 \mu(a df, \Delta):=  a\Delta(f), 
\]
and then applying $\hat{\Tot}$ (noting that $\Tot^{\Pi} \Omega_A^p = \hat{\Tot}\Omega_A^p $).

The proof of \cite[Lemma \ref{poisson-keylemma}]{poisson} ensures that this becomes a chain map (and hence an $R$-CDGA morphism)
\[
 \mu(-,\Delta) \co \DR(A) \to (\hat{\Tot} B, (\pd \pm\delta)_B +\Delta).
\]
\end{definition}

\begin{definition}\label{muwdef}
Given a choice $w \in \Levi_{\GT}(\Q)$ of Levi decomposition for $\GT$ and  $\Delta \in  Q\cP(A,0)_w/G^j$ define 
\[
 \mu_w(-,\Delta) \co \DR(A)\llbracket\hbar\rrbracket/\hbar^j \to T_{\Delta}Q\widehat{\Pol}_w(A,0)/G^j
\]
by applying Definition \ref{mudef} to the stacky CDGAs
\[
 T_0Q\widehat{\Pol}(A,0):= (\prod_{i=0}^kF_i p_w\uline{\CCC}^{\bt}_R(A)\hbar^{i})/G^j
 \]
and the derivation $[\Delta, -]$, then taking the limit over all $k$.
Observe that this map preserves the filtration $\tilde{F}$.
\end{definition}

\begin{definition}\label{Qcompatdef}
We say that a generalised    pre-symplectic structure $\omega$ and an $E_1$ quantisation $\Delta$   are  $w$-compatible (or a $w$-compatible pair) if 
\[
 [\mu_w(\omega, \Delta)] = [-\pd_{\hbar^{-1}}(\Delta)] \in  \H^1(\tilde{F}^2T_{\Delta}Q\widehat{\Pol}_w(A,0)) \cong \H^1(\tilde{F}^2T_{\Delta}Q\widehat{\Pol}(A,0)),
\]
where $\sigma=-\pd_{\hbar^{-1}}$ is the canonical tangent vector of Definition \ref{Qsigmadef}. 
\end{definition}

\begin{definition}\label{vanishingdef}
Given a simplicial set $Z$, an abelian group object $A$ in simplicial sets over $Z$,  a space $X$ over $Z$ and a morphism  $s \co X \to A$ over $Z$, define the homotopy vanishing locus of $s$ over $Z$ to be the homotopy limit of the diagram
\[
\xymatrix@1{ X \ar@<0.5ex>[r]^-{s}  \ar@<-0.5ex>[r]_-{0} & A \ar[r] & Z}.
\]
\end{definition}

\begin{definition}\label{Qcompdef}
Define the space $Q\Comp_w(A,0)$ of $w$-compatible quantised $0$-shifted pairs to be the homotopy vanishing locus of  
\[
 (\mu_w - \sigma) \co G\PreSp(A,0) \by Q\cP_w(A,0) \to TQ\cP_w(A,0)
\]
over $Q\cP_w(A,0)$

We define a cofiltration on this space by setting $ Q\Comp_w(A,0)/G^j$ to be the homotopy vanishing locus of  
\[
 (\mu_w - \sigma) \co (G\PreSp(A,0)/G^j)  \by (Q\cP_w(A,0)/G^j)  \to TQ\cP_w(A,0)/G^j 
\]
over $Q\cP_w(A,0)/G^j $.

\end{definition}

When $j=1$, note that this recovers the notion of compatible $0$-shifted pairs from \cite[\S \ref{poisson-Artincompat}]{poisson}.

\begin{definition}
 Define $Q\Comp_w(A,0)^{\nondeg} \subset Q\Comp_w(A,0)$ to consist of $w$-compatible quantised pairs $(\omega, \Delta)$ with $\Delta$ non-degenerate. This is a union of path-components, and by \cite[Lemma \ref{poisson-compatnondeg}]{poisson} any pre-symplectic form compatible with a non-degenerate quantisation is symplectic, so there is a natural projection 
\[
 Q\Comp_w(A,0)^{\nondeg}\to G\Sp(A,0)
\]
as well as the canonical map
\[
 Q\Comp_w(A,0)^{\nondeg} \to Q\cP_w(A,0)^{\nondeg}.
\]
\end{definition}

\subsection{The equivalences}

\begin{proposition}\label{QcompatP1} 
For any Levi decomposition $w$ of $\GT$,  the canonical map
\begin{eqnarray*}
    Q\Comp_w(A,0)^{\nondeg} \to  Q\cP_w(A,0)^{\nondeg}\simeq  Q\cP(A,0)^{\nondeg}          
\end{eqnarray*}
 is a weak equivalence. In particular, there is a morphism
\[
  Q\cP(A,0)^{\nondeg} \to G\Sp(A,0)
\]
in the homotopy category of simplicial sets.
\end{proposition}
\begin{proof}
We adapt the proof of \cite[Proposition \ref{poisson-compatP1}]{poisson}.
For any $\Delta \in Q\cP_w(A,0)$, the homotopy fibre of $Q\Comp_w(A,0)^{\nondeg} $ over $\Delta$ is just the homotopy fibre of
\[
\mu_w(-,\Delta)  \co G\PreSp(A,0)  \to T_{\Delta}Q\cP_w(A,0) 
\]
over $-\pd_{\hbar^{-1}}(\Delta)$.

The map $\mu_w(-,\Delta) \co \DR(A)\llbracket\hbar\rrbracket \to T_{\Delta}Q\widehat{\Pol}_w(A,0)$ is a morphism of complete $\tilde{F}$-filtered $R\llbracket\hbar\rrbracket$-CDGAs by the proof of \cite[Lemma \ref{poisson-keylemma}]{poisson}. Since the morphism is $R\llbracket\hbar\rrbracket$-linear, it maps $G^k\tilde{F}^p\DR(A)\llbracket\hbar\rrbracket$ to $   G^k\tilde{F}^pT_{\Delta}Q\widehat{\Pol}_w(A,0)$. Non-degeneracy of $\Delta_2$ modulo $F^1$ implies that $\mu_w(-,\Delta)$ induces  quasi-isomorphisms
\[
  \Tot^{\Pi}\Omega^{p-k}\hbar^{k}[k-p] \to \hat{\HHom}_A(\Omega^{p-k}_{A}, A)\hbar^{p-k}[k-p]
\]
on the associated gradeds $\gr_G^k\gr_{\tilde{F}}^p$.  We therefore have a quasi-isomorphism of bifiltered complexes, so we have isomorphisms on homotopy groups:
\begin{eqnarray*}
 \pi_jG\PreSp(A,0)  &\to& \pi_jT_{\Delta}Q\cP(A,0)\\
 \H^{2-j}(\tilde{F}^2 \DR(A)\llbracket\hbar\rrbracket) &\to&  \H^{2-j}(\tilde{F}^2T_{\Delta}Q\widehat{\Pol}(A,0)).
\end{eqnarray*}
\end{proof}

\begin{proposition}\label{quantprop}
For any Levi decomposition $w$ of $\GT$, the   maps
\begin{align*}
 Q\cP_w(A,0)^{\nondeg}/G^j &\to (Q\cP_w(A,0)^{\nondeg}/G^2)\by^h_{(G\Sp(A,0)/G^2)}(G\Sp(A,0)/G^j) \\ 
&\simeq (Q\cP_w(A,0)^{\nondeg}/G^2)\by \prod_{2 \le i<j } \mmc(\DR(A)\hbar^i[1])
\end{align*}
coming from Proposition \ref{QcompatP1}  are weak equivalences for all $j \ge 2$.
\end{proposition}
\begin{proof}
 The proof of \cite[Proposition \ref{DQvanish-quantprop}]{DQvanish} generalises to this setting. We have a commutative diagram 
\[
\begin{CD}
  (Q\Comp_w(A,0)/G^{j+1})_{(\omega, \pi)}@>>> (Q\Comp_w(A,0)/G^j)_{(\omega,\pi)} @>>> \mmc(N(\omega,\pi,j)[2])\\
@VVV @VVV @VVV \\
(G\PreSp(A,0)/G^{j+1})_{\omega}@>>> (G\PreSp(A,0)/G^{j})_{\omega} @>>> \mmc(F^{2-j}\hbar^{j}\DR(A)[2])
\end{CD}
\]
of fibre sequences, with   $N(\omega,\pi, j)$ the cocone of the map
\begin{align*}
F^{2-j}\DR(A)\hbar^{j}
\oplus 
(F^{2-j}\widehat{\Pol}(A,0)\hbar^{j}, \delta_{\pi}) &\to F^{2-j}T_{\pi}\widehat{\Pol}(A,0)\hbar^{j}
\end{align*}
given by combining                                
\begin{align*}
   \mu(-,\pi) \co F^{2-j}\DR(A)\hbar^{j}  &\to F^{2-j}T_{\pi}\widehat{\Pol}(A,0)\hbar^{j}
\end{align*}
with 
\begin{align*}
 \nu(\omega, \pi) + \pd_{\hbar^{-1}} \co F^{2-j}T_{\pi}\widehat{\Pol}(A,0)\hbar^{j-1} &\to F^{2-j}T_{\pi}\widehat{\Pol}(A,0)\hbar^{j}.
\end{align*}
Here $\nu(\omega, \pi)$ is the tangent map of $\mu(\omega, -)$ at $\pi$, given by $\mu(\omega, \pi+ \rho \eps) = \mu(\omega, \pi) + \nu(\omega, \pi)(\rho)\eps$ with $\eps^2=0$.

As in \cite[Lemma \ref{DQvanish-tangentlemma}]{DQvanish}, on the associated graded piece 
\[
 \gr_F^p\widehat{\Pol}(A,0)\hbar^{j} = \hat{\HHom}_A(\Omega^p_{A},A)\hbar^{p+j-1},
\]
the map $\nu(\omega, \pi)$ is given by $p\L^p(\pi^{\sharp} \circ \omega^{\sharp})\hbar$, while $\pd_{\hbar^{-1}}= (1-j-p)\hbar$. Since $\pi$ is non-degenerate, $ \pi^{\sharp} \circ \omega^{\sharp}$ is homotopic to  $1$, so $ \gr_F^p(\nu(\omega, \pi) + \pd_{\hbar^{-1}})$ is homotopic to $(1-j)\hbar$. As this is an isomorphism for all $j \ge 2$, 
 the map  $N(\omega, \pi,j) \to F^{2-j}\DR(A)\hbar^{j}$ is quasi-isomorphism, which inductively gives the required weak equivalences from the fibre sequences above.
\end{proof}

\begin{remark}\label{quantrmk}
Taking the limit over all $j$, Proposition \ref{quantprop}  gives an equivalence
\[
  Q\cP_w(A,0)^{\nondeg} \simeq (Q\cP_w(A,0)^{\nondeg}/G^2)\by \prod_{i \ge 2} \mmc(\DR(A)\hbar^i[1]);
\]
in particular, this means that there is a canonical map 
\[
 (Q\cP(A,0)^{\nondeg}/G^2) \to Q\cP(A,0)^{\nondeg},
\]
dependent on $w$, corresponding to the distinguished point $0 \in \mmc( \hbar^2\DR(A)\llbracket\hbar\rrbracket)$.

Thus to quantise a non-degenerate $0$-shifted Poisson structure $\pi =\sum_{j \ge 2} \pi_j$ (or equivalently, by \cite[Corollary \ref{poisson-compatcor2}]{poisson}, a $0$-shifted symplectic  structure), it suffices to lift the power series $\sum_{j \ge 2} \pi_j \hbar^{j-1}$  to a Maurer--Cartan element of $\prod_{j \ge 2} (F_j\CCC^{\bt}_R(A)/F_{j+2})\hbar^{j-1}$.

Even if $\pi$ is degenerate, a variant of Proposition \ref{quantprop} still holds. Because $\pi^{\sharp} \circ \omega^{\sharp}$ is homotopy idempotent, the map   $\gr_F^p\nu(\omega, \pi)$ has eigenvalues in the interval $[0,p]$, so we just replace  $(1-j)$ with an operator  having eigenvalues in the interval $[1-p-j, 1-j]$.
 Since this is still a quasi-isomorphism for $j>1$, we have
\[
 Q\Comp_w(A,0) \simeq (Q\Comp_w(A,0)/G^2)\by \prod_{i \ge 2} \mmc(\DR(A)\hbar^i).
\]
giving a sufficient first-order criterion  for degenerate quantisations to exist.
\end{remark}

\begin{remark}\label{quantLevirmk}
As in Remark \ref{Levirmk}, we could consider the space $ \oR \Levi_{\GT}(R)$ of $R$-linear Levi decompositions, and the proof of Proposition \ref{quantprop} then gives equivalences
\begin{align*}
&\oR \Levi_{\GT}(R) \by   Q\cP(A,0)^{\nondeg}/G^j\\
 &\to \oR \Levi_{\GT}(R) \by  (Q\cP(A,0)^{\nondeg}/G^2)\by^h_{(G\Sp(A,0)/G^2)}(G\Sp(A,0)/G^j) \\ 
&\simeq \oR \Levi_{\GT}(R) \by   (Q\cP(A,0)^{\nondeg}/G^2)\by \prod_{2 \le i<j } \mmc(\DR(A)\hbar^i[1])
\end{align*}
over $\oR \Levi_{\GT}(R)$. 
\end{remark}

\subsubsection{Self-duality}

\begin{theorem}\label{quantpropsd}
For any Levi decomposition $w$ of $\GT$, there is a canonical weak equivalence
\[
  Q\cP(A,0)^{\nondeg,sd} \simeq \cP(A,0)^{\nondeg} \by \mmc(\hbar^2 \DR(A)\llbracket\hbar^2\rrbracket[1]).
\]
In particular, $w$ gives  a canonical choice of self-dual quantisation for any non-degenerate $0$-shifted Poisson structure  on $A$.
\end{theorem}
\begin{proof}
Lemma \ref{filtsd} implies that we have weak equivalences
\begin{align*}
 Q\cP(A,0)^{sd}/G^{2i} &\to Q\cP(A,0)^{sd}/G^{2i-1}\\
Q\cP(A,0)^{sd}/G^{2i+1} &\to (Q\cP(A,0)^{sd}/G^{2i})\by^h_{(Q\cP(A,0)/G^{2i})}(Q\cP(A,0)/G^{2i+1}).
\end{align*}

Combined with Proposition \ref{quantprop}, the latter  gives  weak equivalences
\[
 Q\cP(A,0)^{\nondeg,sd}/G^{2i+1} \to (Q\cP(A,0)^{\nondeg,sd}/G^{2i})\by \mmc(\hbar^{2i} \DR(A)[1])
\]
for all $i>0$, so
\begin{align*}
 Q\cP(A,0)^{\nondeg,sd}/G^{2i+1}&\simeq   (Q\cP(A,0)^{\nondeg,sd}/G^{2i})\by \mmc(\hbar^{2i} \DR(A)[1])\\
& \simeq Q\cP(A,0)^{\nondeg}/G^{2i-1}\by \mmc(\hbar^{2i} \DR(A)[1]),
\end{align*}
and we have seen that $*$ acts trivially on  $ Q\cP(A,0)/G^1=\cP(A,0) $, so $ Q\cP(A,0)^{sd}/G^1\simeq \cP(A,0)$.
\end{proof}

\begin{example}
When applied to the polynomial ring $A=R[t_1, \ldots, t_d]$ concentrated in degree $0$, Theorem \ref{quantpropsd} implies  that the map  $Q\cP(A,0)^{\nondeg,sd} \to \cP(A,0)^{\nondeg}$ has simply connected fibres, via 
vanishing of de Rham cohomology. $2$-automorphisms are given by $\exp(\hbar R\brhh)= \{ r(\hbar) \in 1+ \hbar R \brh ~:~  r(\hbar)r(-\hbar)=1\}$, with $\hbar^2s(\hbar^2) \in \hbar^2R\brhh \cong \hbar^2\H^0\DR(A)\brhh$ corresponding under Theorem \ref{quantpropsd} to the $2$-automorphism $\exp(\int s(\hbar^2)d\hbar)$.

In detail, for a fixed non-degenerate Poisson structure $\pi$, self-dual quantisations $(A\brh, \star_{\hbar})$ (with involution $a(\hbar)^* := a(-\hbar)$)  are unique up to involutive isomorphism (i.e. $\theta$ with $\theta(a^*)=\theta(a)^*$).  Those isomorphisms are unique up to involutive inner automorphism (i.e. conjugation by $\{ a(\hbar) \in 1+ \hbar A \brh ~:~  a(\hbar)\star_{\hbar}a(-\hbar)=1\}$) and the inner automorphism $a(\hbar)$ (regarded as a $2$-morphism) is unique up to multiplication by $\exp(\hbar R\brhh)$.
\end{example}

\begin{remark}\label{oddcoeffsrmk}
The proof of Theorem  \ref{quantpropsd}  shows that for a  self-dual quantisation of a non-degenerate $0$-shifted Poisson structure, the $w$-compatible generalised symplectic structure is determined by its even coefficients. This raises the question of whether  the odd coefficients must be homotopic to $0$, as happens in the $(-1)$-shifted case by  \cite[Remark \ref{DQvanish-oddcoeffsrmk}]{DQvanish}. The answer depends on the choice of $w$, as follows.

The involution $i$ from Lemma \ref{involutiveHH} is not just a DGLA automorphism. If we write $f^t:= -i(f)$, then $(f\cdot g)^t= (-1)^{\deg f\deg g} g^t\cdot f^t$ and $\{f\}\{g_1, \ldots, g_m\}^t \simeq \mp \{f^t\}\{g_m^t, \ldots , g_1^t\}$, so $(-)^t\co \CC^{\bt}_R(A)^{\op} \to \CC^{\bt}_R(A)$ makes $\CC^{\bt}_R(A)$ into an anti-involutive brace algebra. The opposite brace algebra $B^{\op}$ is most easily understood in terms of the associated $B_{\infty}$-algebra, which is a bialgebra structure on the tensor coalgebra $T(B[1])$: to form $B^{\op}$, we just take the opposite  comultiplication on $T(B[1])$.

We can define an involution of the $E_2$ operad similarly, which takes an embedding $[1,k] \by I^2 \to I^2$ of $k$ little squares in a big square, and reverses the order of the labels $[1,k]$ with appropriate signs. This involution comes from an element $t \in \GT$ which maps to $-1 \in \bG_m$. It gives a notion of opposite $E_2$-algebra, with $(-)^t\co \CC^{\bt}_R(A)^{\op} \to \CC^{\bt}_R(A) $ then giving $\CC^{\bt}_R(A) $ the structure of an anti-involutive $E_2$-algebra.

Levi decompositions $w$ of $\GT$ with $w(-1)=t$ form a torsor $\Levi_{\GT}^t$  for the subgroup $(\GT^1)^t$ of $t$-invariants in $\GT^1$. (To see that $\Levi_{\GT}^t$ is non-empty, first pick any Levi decomposition $w_0$, and write $w_0(-1)=tu$ for $u\in \GT^1$. Since $t$ and $w_0(-1)$ are both of order $2$, we have $u= \ad_t(u^{-1})$, so $u^{\half}=\ad_t(u^{-\half})$, giving $w:=\ad_{u^{-\half}}\circ w_0 \in \Levi_{\GT}^t$.) Under the isomorphism $\Ass^1 \cong \Levi_{\GT}$ between associators and Levi decompositions, elements of $\Levi_{\GT}^t$ correspond to even associators.  

For any such $w\in \Levi_{\GT}^t(\Q)$, the $\infty$-functor $p_w$ sends  opposite $E_2$-algebras to  opposite $P_2$-algebras, defined by reversing the sign of the Lie bracket. This gives  
\[
 \mu_w(\omega, \Delta)^t= \mu_w(\omega, -\Delta^t),
\]
so $\omega(\hbar)$ is compatible with $\Delta$ if and only if $\omega(-\hbar)$ is compatible with $\Delta^*$, implying that the odd coefficients of  $\omega$ must be homotopic to $0$ when $\Delta$ is non-degenerate and self-dual.

For a more explicit description of the generalised symplectic structure $\omega$ corresponding to a non-degenerate self-dual quantisation $\Delta$,  observe that we then have an isomorphism
\begin{align*}
 \mu_w(-, \Delta) \co  &\H^*( F^2\DR(X) \by  \hbar^2\DR(X)\llbracket\hbar^2\rrbracket) \\
&\to \{v \in \H^*(T_{\Delta}(\tilde{F}^2Q\widehat{\Pol}(A,0))~:~ v(-\hbar)= v^t(\hbar)\},
\end{align*}
and that $[\omega]$ must be  the inverse image of $[\hbar^2 \frac{\pd \Delta}{\pd \hbar}]$.

\end{remark}

\begin{remark}\label{quantsdLevirmk}
Similarly to Remarks \ref{Levirmk} and \ref{quantLevirmk}, we could consider the space $ \oR \Levi_{\GT}^t(R)$ of $R$-linear Levi decompositions with $w(-1)=t$, and the proof of Theorem  \ref{quantpropsd} then combines with  Remark \ref{oddcoeffsrmk} to give a canonical weak equivalence
\[
  \oR \Levi_{\GT}^t(R) \by  Q\cP(A,0)^{\nondeg,sd} \simeq  \oR \Levi_{\GT}^t(R) \by \cP(A,0)^{\nondeg} \by \mmc(\hbar^2 \DR(X)\llbracket\hbar^2\rrbracket[1]).
\]
over $  \oR \Levi_{\GT}^t(R)$.
\end{remark}

\subsection{Comparison with Kontsevich--Tamarkin quantisations}

In \cite{kontsevichDQAlgVar}, Kontsevich showed that for a smooth algebraic variety $X$  over a field $k$ of characteristic $0$,  every Poisson structure $\pi$ lifts to an algebroid quantisation of $X$. We now investigate how this quantisation relates to our quantisations above when $\pi$ is non-degenerate and $X$ affine; by descent, this comparison will extend to the global quantisations of the next section. 
Unlike \cite{kontsevichPoisson}, the approach of \cite{kontsevichDQAlgVar}  does not start from  a specific local quantisation, instead  giving a construction dependent on a  choice of explicit quantisation formula over $k$, which is stated to depend on  a choice of Drinfeld associator with coefficients in $k$. 

Tamarkin's approach \cite{tamarkinOperadicKontsevichFormality} to quantisation is more suited to comparison with our constructions; although it is formulated for manifolds, it can also be adapted to algebraic varieties compatibly with \cite{kontsevichDQAlgVar}, as indicated in  \cite[Remark 8.2.1]{vdBerghGlobalDQ}. It relies on the choice of a Drinfeld associator (or equivalently on a Levi decomposition $w$ of $\GT$). As in  \cite{kontsevichOperads} or \cite[proof of Theorem 9.5.1]{vdBerghGlobalDQ}, the key is the existence of a canonical quasi-isomorphism 
\[
\phi_w \co  \prod_{i \ge 0}\Hom_A(\Omega^i_{A},A)[-i]\simeq p_w\CCC^{\bt}_k(A)
\]
 of filtered  $P_{2, \infty}$-algebras, lifting the HKR isomorphism. 

The quasi-isomorphism gives  a $k\llbracket \hbar\rrbracket$-linear  $P_{2, \infty}$-algebra quasi-isomorphism
\[
 \widehat{\Pol}(A,0)\llbracket \hbar \rrbracket\simeq Q\widehat{\Pol}_w(A,0)
\]
sending the filtration  $ \{ \prod_i  F^{p-i}\hbar^{i} \}_p$ on the left to $\{\tilde{F}^p\}_p$, and inclusion of $\widehat{\Pol}(A,0)$ on the left then gives rise to the quantisation map $\phi_w \co \cP(A,0) \to Q\cP_w(A,0)$ on Maurer--Cartan spaces.

This allows us to make a comparison with our constructions:
\begin{lemma}\label{compKTlemma}
For $A$ smooth over a field $k \supset \Q$ and $w \in \Levi_{\GT}(k)$,  the map $\phi_w \co \cP(A,0) \to Q\cP_w(A,0)$ from Poisson structures to quantisations, following   \cite[Remark 8.2.1 and Theorem 9.5.1]{vdBerghGlobalDQ}, extends to map
\begin{align*}
   \Comp(A,0) &\to   Q\Comp_w(A,0),\\
(\omega, \pi) &\mapsto (\omega, \phi_w(\pi))
\end{align*}
from compatible pairs to $w$-compatible pairs.
\end{lemma}
\begin{proof}
  Functoriality of $\mu$ implies that $\mu_w(\omega, \phi_w(\pi))= \phi_w(\mu(\omega, \pi))$, so $\phi_w(\pi)$ is $w$-compatible with a pre-symplectic form $\omega$ whenever  $\pi$ is compatible with $\omega$.
\end{proof}

When $w$ comes from an even associator, we then have:
\begin{proposition}\label{cfKT}
For a smooth algebra $A$ over a field $k \supset \Q$ and for $w \in \Levi_{\GT}^t(k)$,
 the map  $\phi_w$ restricts to a map $ \phi_w \co \cP(A,0)^{sd} \to Q\cP_w(A,0)^{sd}$, i.e. the Tamarkin quantisation $\phi_w(\pi)$ of any Poisson structure $\pi$ on $A$ is self-dual. When $\pi$ is non-degenerate, the  quantisation $\phi_w(\pi)$ corresponds under Theorem \ref{quantpropsd} to the constant de Rham  power series $\pi^{-1}$.
\end{proposition}
\begin{proof}
The global formality quasi-isomorphisms of \cite{vdBerghGlobalDQ} depend only on a choice of quasi-isomorphism 
in the  formal case, i.e. replacing $A$ with  the pro-algebra $k\llbracket t_1, \ldots, t_d\rrbracket$ when $A$ has dimension $d$. Tamarkin's approach to quantisation, as described in \cite{kontsevichOperads}, relies on showing that the equivalence class of $P_2$-algebra deformations of $\Pol(k[ t_1, \ldots, t_d],0)$ 
invariant under affine transformations is trivial. The same is true for the equivalence class  of  anti-involutive $P_2$-algebra deformations, replacing the deformation complex of  \cite[\S 3.4]{kontsevichOperads} with its subspace of odd weight. 

When $w \in \Levi_{\GT}^t(k)$, the involution $i$ of Lemma \ref{involutiveHH} gives an anti-involution $-i$ on the  $P_2$-algebra  $p_w\CCC^{\bt}_k(A)$, and the argument above shows that the map $\phi_w$ is compatible with the involutions, so we have $\phi_w \co \widehat{\Pol}(A,0)\llbracket \hbar^2 \rrbracket \simeq Q\widehat{\Pol}_w(A,0)^{sd}$, giving the restriction claimed.

The map of  Lemma \ref{compKTlemma} then restricts to give a morphism  $\Comp(A,0) \to   Q\Comp_w(A,0)^{sd}$, 
and further restriction to non-degenerate elements gives the correspondence between  $\phi_w(\pi)$ and $\pi^{-1}$ via Theorem \ref{quantpropsd}.
\end{proof}

\begin{remark}\label{nonfedosovrmk}
Extending Theorem \ref{quantpropsd} to give existence of quantisations for degenerate Poisson structures on more general stacky CDGAs $A$ requires  an alternative  to \cite{kontsevichDQAlgVar}. In \cite{DQpoisson}, this is established for finitely presented chain CDGAs (and hence derived Deligne--Mumford stacks). Instead of looking at quantisations of $k\llbracket t_1, \ldots, t_d\rrbracket$,  the problem is rigidified there by observing that $p_w\CCC^{\bt}_k(A)$ is an \emph{involutive} filtered deformation of the $P_2$-algebra $\widehat{\Pol}(A,0)$ whenever  
$w$ is even.
 Calculations based on the method of  \cite{poisson} then show that the $\infty$-groupoid of  deformations of  $\widehat{\Pol}(A,0)$ as an anti-involutive filtered $P_2$-algebra is contractible, giving the desired quasi-isomorphism $p_w\CCC^{\bt}_k(A) \simeq \widehat{\Pol}(A,0)$.
\end{remark}

\section{Quantisation for derived  stacks}\label{stacksn}  

As in \cite[\S \ref{DQvanish-Artinsn}]{DQvanish}, in order to pass from stacky CDGAs to derived Artin stacks, we will exploit \'etale functoriality using Segal spaces.

\subsection{Quantised polyvectors  for diagrams}\label{Artindiagsn}

\begin{definition} 
Given a small category $I$,  an $I$-diagram $A$ in $DG^+dg\CAlg(R)$, and an  $A$-module $M$ in $I$-diagrams of chain cochain complexes, define the filtered Hochschild cochain complex $\CCC^{\bt}_R(A,M)$ to be the equaliser of the obvious diagram
\[
\prod_{i\in I} \CCC^{\bt}_R(A(i),M(i)) \implies \prod_{f\co i \to j \text{ in } I} \CCC^{\bt}_R(A(i),M(j)),
\]
with the filtration $F_k\CCC^{\bt}_R(A,M)$ defined similarly, for the Hochschild complexes of Definition \ref{HHdef}.

We then write $\CCC^{\bt}_R(A):= \CCC^{\bt}_R(A,A)$,   which  inherits the structure of  a brace algebra from  each $\CCC^{\bt}_R(A(i),A(i))$. 
\end{definition}

For $f \co i \to j$ a morphism in $I$, observe that the HKR  maps 
\[
 \gr^F_{k}\CCC^{\bt}_R(A(i),f_*M(j))   \to \hat{\HHom}_{A(i)}(\Omega^k_{A(i)},f_*M(j))
\]
are quasi-isomorphisms whenever $A(i)$ is  cofibrant in the model structure of Lemma \ref{bicdgamodel}. 
Also note that if $u \co I \to J$ is a morphism of small categories and $A$ is a functor from  $J$ to  $DG^+dg\CAlg(R)$ with $B= A \circ u$, then we have a natural map $ \CCC^{\bt}_R(A) \to \CCC^{\bt}_R(B)$.

In order to ensure that $\CCC^{\bt}_R(A,M)$ has the correct homological properties, we now consider categories  of  the form $[m]= (0 \to 1 \to \ldots \to m)$.
\begin{lemma}\label{calcCClemma}
 If $A$ is an  $[m]$-diagram in  $DG^+dg\CAlg(R)$  which is cofibrant and  fibrant for the injective model structure (i.e. each $A(i)$ is cofibrant in the model structure of Lemma \ref{bicdgamodel} and the maps $A(i) \to A(i+1)$ are surjective), then
$
 \gr^F_{k}\CCC^{\bt}_R(A)
$
is a model for the derived $\Hom$-complex $\oR \hat{\HHom}_A(\oL\Omega^k_A,A)$. 
\end{lemma}
\begin{proof}
The proof of  \cite[Lemma \ref{poisson-calcTOmegalemma}]{poisson} adapts verbatim to stacky CDGAs to give   $\hat{\HHom}_A(\Omega^k_{A},A) \simeq \oR\hat{\HHom}_A(\oL\Omega^k_A,A)$, from which this result follows immediately via the HKR isomorphism. 
 \end{proof}

The constructions in \S \ref{defsn} now all carry over verbatim, generalising from cofibrant stacky CDGAs to $[m]$-diagrams of stacky CDGAs which are cofibrant and  fibrant for the injective model structure. In order to identify $Q\cP/G^1$ with  $\cP$, and for  notions such as non-degeneracy to make sense, we have to  assume that for our fibrant cofibrant  $[m]$-diagram $A$ of stacky CDGAs, each $A(j)$ satisfies Assumption \ref{biCDGAprops}, so there exists $N$ for which the chain complexes $(\Omega^1_{A(j)}\ten_{A(j)}A(j)^0)^i $ are acyclic for all $i >N$.

\begin{definition}\label{ICompdef}
Given an $[m]$-diagram $A$ satisfying the conditions of Lemma \ref{calcCClemma},   define  
\[
 G\PreSp(A,0):=  G\PreSp(A(0),0)= \Lim_{i\in [m]}  G\PreSp(A(i),0),
\]
for the space $G\PreSp$ of generalised pre-symplectic structures of Definition \ref{GPreSpdef}. 

 Given a choice $w \in \Levi_{\GT}(\Q)$ of Levi decomposition for $\GT$, define 
\[
 \mu_w \co G\PreSp(A,0) \by Q\cP_w(A,0) \to TQ\cP_w(A,0) 
\]
by setting $\mu_w(\omega, \Delta)(i):= \mu_w(\omega(i), \Delta(i)) \in TQ\cP_w(A(i),0)$ for $i \in [m]$, and let $ Q\Comp_w(A,0)$ be the homotopy vanishing locus of
\[
(\mu_w - \sigma) \co  G\PreSp(A,0) \by Q\cP_w(A,0) \to  TQ\cP_w(A,0).
\]
over $Q\cP_w(A,0)$.
\end{definition}

The following is  \cite[Lemma \ref{poisson-calcTlemma2}]{poisson}:
\begin{lemma}\label{bicalcTlemma2}
If $D=(A\to B)$ is a fibrant cofibrant  $[1]$-diagram in $DG^+dg\CAlg(R)$ which is formally \'etale in the sense that the map
\[
 \{\Tot \sigma^{\le q} (\Omega_{A}^1\ten_{A}B^0)\}_q \to \{\Tot \sigma^{\le q}(\Omega_{B}^1\ten_BB^0)\}_q
\]
is a pro-quasi-isomorphism, then the map    
\[
 \hat{\HHom}_D(\Omega^k_D,D) \to \hat{\HHom}_{A}(\Omega^k_{A},A),
\]
is a quasi-isomorphism for all $k$.
\end{lemma}

As in \cite[\S \ref{poisson-bidescentsn}]{poisson}, for any of the constructions $F$ based on $Q\cP$, 
\cite[Definition \ref{poisson-inftyFdef}]{poisson} adapts  to give
an $\infty$-functor
\[
 \oR F \co \oL DG^+dg\CAlg(R)^{\et} \to  \oL s\Set
\]
with $(\oR F)(A) \simeq F(A)$
for all cofibrant stacky CDGAs $A$ over $R$, where $DG^+dg\CAlg(R)^{\et} \subset DG^+dg\CAlg(R)$ is the subcategory of homotopy formally \'etale morphisms.

Immediate consequences of Propositions \ref{QcompatP1} and \ref{quantprop} and Theorem \ref{quantpropsd} are that for any $w \in \Levi_{\GT}(\Q)$, the canonical maps 
\begin{align*}
\oR Q\Comp_w(-,0)^{\nondeg} &\to \oR Q\cP_w(-,0)^{\nondeg}\simeq \oR Q\cP(-,0)^{\nondeg} \\
\oR Q\cP_w(-,0)^{\nondeg}/G^j &\to (\oR Q\cP_w(-,0)^{\nondeg}/G^2)\by^h_{\oR(G\Sp(-,0)/G^2)}\oR(G\Sp(-,0)/G^j) \\ 
 \oR Q\cP_w(-,0)^{\nondeg,sd} &\simeq \oR\cP(-,0)^{\nondeg} \by \mmc(\hbar^2 \DR(-)\llbracket\hbar^2\rrbracket[1])
\end{align*}
 are weak equivalences of $\infty$-functors on the full subcategory of $\oL DG^+dg\CAlg(R)^{\et}$ consisting of objects satisfying the conditions of Assumption \ref{biCDGAprops}, for all $j \ge 2$.

\subsection{Hypergroupoids}

We now recall the main constructions from \cite{stacks2}, as summarised in \cite[\S \ref{poisson-hgpdsn}]{poisson}.

We  require our chain CDGA $R$ over $\Q$ to be concentrated in non-negative chain degrees, and write   $dg_+\CAlg(R)\subset dg\CAlg(R) $ for the full subcategory of chain CDGAs  which are concentrated in non-negative chain degrees. We denote the opposite category to $dg_+\CAlg(R) $ by $DG^+\Aff_R$. Write $sDG^+\Aff_R$ for the category of simplicial diagrams in $DG^+\Aff_R $. A morphism in $DG^+\Aff_R $ is said to be a fibration if it is given by a cofibration in the opposite category $dg_+\CAlg(R)$.

\begin{definition} 
Given $Y \in sDG^+\Aff_R$, a DG Artin  $N$-hypergroupoid $X$ over $Y$ is a morphism $X \to Y$ in $sDG^+\Aff_R$ for which:
\begin{enumerate}
 \item the matching maps
$$
X_m \to M_{\pd \Delta^m} (X)\by_{M_{\pd \Delta^m} (Y)}Y_m 
$$
are fibrations for all $m\ge 0$;

\item  the   partial matching maps
$$
X_m \to M_{\Lambda^m_k} (X)\by_{M^h_{\Lambda^m_k} (Y)}^hY_m 
$$
are  smooth surjections for all $m \ge 1$ and $k$, and are weak equivalences for all $m>N$ and all $k$.
\end{enumerate}

A morphism $X\to Y$ in $sDG^+\Aff_R$ is a  trivial DG Artin (resp. DM)  $N$-hypergroupoid if and only if the matching maps
$$
X_m \to M_{\pd \Delta^m} (X)\by_{M_{\pd \Delta^m} (Y)}Y_m 
$$
are surjective smooth  fibrations  for all  $m$, and are weak equivalences for all $m\ge n$.
\end{definition}

The following is \cite[Theorem \ref{stacks2-bigthm} and Corollary \ref{stacks2-Dequivcor}]{stacks2}, as spelt out in \cite[Theorem \ref{stacksintro-dbigthm}]{stacksintro}:

\begin{theorem}\label{dbigthm}
%
The $\infty$-category of strongly quasi-compact $N$-geometric derived Artin stacks  over $R$ is given by localising the category  of DG Artin  $N$-hypergroupoids  over $R$  at the class of trivial relative  DG Artin   $N$-hypergroupoids.
\end{theorem}

Given a  DG Artin   $N$-hypergroupoid $X$, we denote the associated $N$-geometric derived Artin  stack by $X^{\sharp}$.
%
%

There is a denormalisation functor $D$ from non-negatively graded CDGAs to cosimplicial algebras, with 
 left adjoint $D^*$ as in \cite[Definition \ref{ddt1-nabla}]{ddt1}. 
Given a cosimplicial chain CDGA $A$, $D^*A$ is then a stacky CDGA, with $ (D^*A)^i_j=0$ for $i<0$. 
%

\subsection{Global quantisations}\label{globalquantnsn}

The following is \cite[Corollary \ref{poisson-gooddescent}]{poisson}, showing that a  DG Artin   $N$-hypergroupoid $X$ can be recovered from the stacky CDGAs $D^*O(X^{\Delta^j}) $; this should be thought of as a resolution by derived Lie algebroids.
\begin{lemma}\label{gooddescent}
For any simplicial presheaf  $F$ on $DG\Aff(R)$ and any Reedy fibrant simplicial derived affine $X$, there is a canonical weak equivalence 
\[
 \ho \Lim_{j\in\Delta} \map( \Spec DD^*O(X^{\Delta^j}), F) \to \map (X, F).
\]
\end{lemma}

%


Lemma \ref{gooddescent} and  \cite[Proposition \ref{poisson-tgtcor2}]{poisson} ensure that if a morphism $X \to Y$ of DG Artin $N$-hypergroupoids becomes an equivalence on hypersheafifying, then $D^*O(Y) \to D^*O(X)$ is formally \'etale in the sense of Lemma \ref{bicalcTlemma2}. In particular this means that the maps $\pd^i \co D^*O(X^{\Delta^j}) \to D^*O(X^{\Delta^{j+1}})$ and $\sigma^i \co   D^*O(X^{\Delta^{j+1}})\to D^*O(X^{\Delta^j})$ are formally \'etale. Thus $D^*O(X^{\Delta^{\bt}})$ can be thought of as a DM hypergroupoid in stacky CDGAs, and we may make the following definition:

\begin{definition}\label{biinftyFXdef}
Given a DG Artin $N$-hypergroupoid $X$ over $R$ and   any of the constructions $F$ based on $Q\cP$,
 write 
\[
 F(X):= \ho\Lim_{j \in \Delta} \oR F(D^* O(X^{\Delta^j})),
\]
for $\oR F$ as in \S \ref{Artindiagsn}.
\end{definition}

The proof of \cite[Proposition \ref{poisson-biinftyFXwell}]{poisson} shows that if $Y \to X$ is a trivial DG Artin hypergroupoid, then the morphism 
$ F(X) \to F(Y)$
is an equivalence for any of the constructions $F= \cP, \Comp, \PreSp$. Thus the following is well-defined:
\begin{definition}\label{QpoissdefX}
 Given a strongly quasi-compact DG Artin $N$-stack $\fX$ over $R$,  define the spaces  $Q\cP(\fX,0)$, $Q\Comp_w(\fX,0)$, $G\Sp(\fX,0)$  to be  the spaces
$
 Q\cP(X,0), \Comp_w(X,0), G\Sp(X,0)
$
for any DG Artin $N$-hypergroupoid $X$ with $X^{\sharp} \simeq \fX$.
\end{definition}

\begin{examples}
Examples of derived stacks $\fX$ with canonical $0$-shifted symplectic structures (elements of  $G\Sp(\fX,0)/G^1$) include the derived moduli stack  $\oR\Perf_S$ of perfect complexes on an algebraic $K3$ surface $S$, or the derived moduli stack $\oR\Loc_G(\Sigma)= \map(\Sigma,BG)$ of locally constant $G$-torsors on a compact oriented topological surface $\Sigma$, for an algebraic  group $G$ equipped with a Killing form on its Lie algebra. These both follow from \cite[\S 3.1]{PTVV}, with the symplectic form in the latter case coming from the $2$-shifted symplectic structure in  $\H^4(F^2\DR(BG))$, via the composition
\[
 F^2\DR(BG)\to F^2\DR(\oR\Loc_G(\Sigma) )\ten_{\Q}\oR\Gamma(\Sigma,\Q)\to  F^2\DR(\oR\Loc_G(\Sigma))[-2]
\]
of pullback along $\Sigma \by \oR\Loc_G(\Sigma) \to BG$  with Poincar\'e duality. 

When $\Sigma$ is the $2$-sphere, the Killing form gives an equivalence $\oR\Loc_G(\Sigma) \simeq T^*BG$, and  
for any derived Artin stack $\fY$, \cite{calaqueShiftedCotangent} gives a $0$-shifted symplectic structure on the derived cotangent stack $T^*\fY$. Example \ref{quantex} generalises to give canonical quantisations in $  Q\cP(T^*\fY,0)$, defined  in terms of differential operators. For explicit hypergroupoid resolutions of $T^*BG$,  the   stacky CDGAs $\oL D^* O((T^*BG)^{\Delta^j})$ featuring in our definition of Poisson structures are just given by $O(T^*[G^j/\g^{j+1}])$ in the notation of Example \ref{quantex}. 
\end{examples}

\begin{proposition}\label{Perprop2}
For any strongly quasi-compact DG Artin $N$-stack $\fX$ over $R$,   there is a natural map  from $Q\cP(\fX,0)$ to the space of curved $A_{\infty}$ deformations $(\per_{dg}(\fX)\llbracket \hbar \rrbracket, \{m^{(i)}\}_{i\ge 0})$ of the dg category $\per_{dg}(\fX) $ of perfect $\O_{\fX}$-complexes.

This restricts to a map from $Q\cP(\fX,0)^{sd}$  to the space of anti-involutive curved $A_{\infty}$ deformations of $\per_{dg}(\fX) $.
\end{proposition}
\begin{proof}
  Combining \cite[Proposition \ref{stacks2-qcohequiv}]{stacks2} with \cite[Lemma \ref{poisson-denormmod} and Corollary \ref{poisson-gooddescent}]{poisson} and choosing a derived Artin $N$-hypergroupoid $X$ representing $\fX$, we have 
\[
 \per_{dg}(\fX) \simeq \ho \Lim_{j\in\Delta} \per_{dg}(D^*O(X^{\Delta^j})).
\]
By definition, $ Q\cP(\fX,0)\simeq \ho \Lim_{j\in\Delta}Q\cP(D^*O(X^{\Delta^j},0)  $, so the existence of the map follows from  Proposition \ref{Perprop}. The second statement is an immediate consequence.
\end{proof}

Adapting \cite[Definition \ref{DQvanish-nondegstack}]{DQvanish} to unshifted structures gives:
\begin{definition}\label{nondegstack}
 Given a Poisson structure $\pi \in \cP(\fX,0)$, we say that $\pi$ is non-degenerate if the induced map
\[
 \pi^{\sharp} \co \oL\Omega^1_{\fX} \to \oR\cHom_{\sO_X}(\oL\Omega^1_{\fX}, \sO_{\fX})
\]
is a quasi-isomorphism of sheaves on $\fX$, and $\oL\Omega^1_{\fX} $ is perfect.
\end{definition}

Combined with  the results   above, an immediate consequence of the generalisation of Propositions \ref{QcompatP1} and \ref{quantprop} and Theorem \ref{quantpropsd} in \S \ref{Artindiagsn} is: 
  \begin{theorem}\label{Artinthm}
For any strongly quasi-compact DG Artin $N$-stack $\fX$ over $R$, and  any $w \in \Levi_{\GT}(\Q)$, there are canonical weak equivalences 
\begin{align*}
Q\Comp_w(\fX,0)^{\nondeg} &\to  Q\cP_w(\fX,0)^{\nondeg}\simeq  Q\cP(\fX,0)^{\nondeg} \\
 Q\cP_w(\fX,0)^{\nondeg}/G^j &\to ( Q\cP_w(\fX,0)^{\nondeg}/G^2)\by^h_{(G\Sp(\fX,0)/G^2)}(G\Sp(\fX,0)/G^j) \\ 
  Q\cP_w(\fX,0)^{\nondeg,sd} &\simeq \cP(\fX,0)^{\nondeg} \by \mmc(\hbar^2 \DR(\fX)\llbracket\hbar^2\rrbracket[1]).
\end{align*}
\end{theorem}
This establishes the existence of $0$-shifted deformation quantisations as conjectured in \cite[\S 5.3]{toenICM}, bypassing \cite[Conjecture 5.3]{toenICM} which would also allow quantisation of degenerate Poisson structures (but see Remark \ref{nonfedosovrmk}). 

\begin{remarks}
 As in Remark \ref{quantLevirmk}, instead of choosing one Levi decomposition, we could work with the space $ \oR \Levi_{\GT}(R)$ of $R$-linear Levi decompositions, and the proof of Theorem \ref{Artinthm} then gives equivalences
\begin{align*}
&\oR \Levi_{\GT}(R) \by   Q\cP(\fX,0)^{\nondeg}/G^j\\
 &\to \oR \Levi_{\GT}(R) \by  (Q\cP(\fX,0)^{\nondeg}/G^2)\by^h_{(G\Sp(\fX,0)/G^2)}(G\Sp(\fX,0)/G^j) ,\\ 
&\oR \Levi_{\GT}(R) \by Q\cP(\fX,0)^{\nondeg,sd}\\
 &\simeq \oR \Levi_{\GT}(R) \by \cP(\fX,0)^{\nondeg} \by \mmc(\hbar^2 \DR(\fX)\llbracket\hbar^2\rrbracket[1])
\end{align*}
over $\oR \Levi_{\GT}(R)$. 

As in Remark \ref{oddcoeffsrmk}, the power series $w$-compatible with a quantisation, although determined by its even coefficients, might have odd coefficients unless $w(-1)=t$. The reasoning of  Remark \ref{quantsdLevirmk} gives a canonical equivalence
\[
 \oR \Levi_{\GT}^t(R) \by Q\cP(\fX,0)^{\nondeg,sd} \simeq \oR \Levi_{\GT}^t(R) \by \cP(\fX,0)^{\nondeg} \by \mmc(\hbar^2 \DR(\fX)\llbracket\hbar^2\rrbracket[1])
\]
over $\oR \Levi_{\GT}^t(R)$ which does send a quantisation to its family of compatible de Rham power series.

As in Remark \ref{sdgerbermk}, self-dual quantisations of $\O_{\fX}$ also give rise to self-dual quantisations of all anti-involutive $\bG_m$-gerbes, and in particular of the Picard algebroid with anti-involution given by $\oR\hom_{\O_{\fX}}(-,\sL)$ for any line bundle $\sL$. 

Finally, applying \'etale descent to Proposition \ref{cfKT} shows that for a smooth DM stack $X$, the Kontsevich--Tamarkin quantisation $\phi_w(\pi)$ of any Poisson structure $\pi$ on $X$ is self-dual  whenever $w \in \Levi_{\GT}^t(k)$.  When $\pi$ is non-degenerate, the  quantisation $\phi_w(\pi)$ then corresponds under Theorem \ref{Artinthm} to the constant de Rham power series $\pi^{-1}$. 
\end{remarks}

\bibliographystyle{alphanum}
\bibliography{references.bib}
\end{document}